\documentclass{elsarticle}

\pdfoutput=1

\usepackage{hyperref}
\usepackage{graphicx}
\usepackage{amssymb}
\usepackage{amsfonts}
\usepackage{amsmath}
\usepackage{amsthm}
\usepackage{amsxtra}
\usepackage{alltt}
\usepackage{natbib}
\usepackage{stmaryrd}
\usepackage{setspace}
\usepackage{array}
\usepackage{mathtools}
\usepackage{color}
\usepackage{bm}
\usepackage{framed}
\usepackage{subcaption}
\usepackage{cleveref}
\usepackage{cancel}

\newtheorem{theorem}{Theorem}

\newtheorem{lemma}{Lemma}

\def\wt#1{\widetilde{#1}}
\def\ol#1{\overline{#1}}
\def\red#1{\textcolor{red}{#1}}
\def\blue#1{\textcolor{blue}{#1}}

\def\mb#1{\mathbf{#1}}

\newcommand{\mbs}[1]{\bm{\mathsf{ #1}}}
\def\fnto#1{\overline{\mbs{#1}}}
\def\fntb#1{\mbs{#1}}
\def\fnts#1{\boldsymbol{#1}}
\def\paren#1{\left({#1}\right)}
\def\DEF{:=}

\newcommand{\mat}[2] [ccccccccc]{ \left( \begin{array}{#1} #2\\ \end{array}\right) }
\newcommand{\eval}[2]{\biggr|_{#1}^{#2}}

\newcommand{\MAT}[9]{
\left(\begin{array}{|cc|c|cc|c}
	 \multicolumn{1}{c}{#1} \\ \cline{1-3}
	 #2 & \multicolumn{1}{c}{#3} &  \\ 
	 #4 & \multicolumn{1}{c}{#5} & #1 \\ \cline{1-5}
	  \multicolumn{1}{c}{}&& \multicolumn{1}{c}{#6}& #7 & \\
	  \multicolumn{1}{c}{}&& \multicolumn{1}{c}{#8}& #9 & #1 \\ \cline{3-5}
	  \multicolumn{1}{c}{}&\multicolumn{1}{c}{}&\multicolumn{1}{c}{} & & \multicolumn{1}{c}{} & \ddots
\end{array}\right)
}

\newcommand{\MATT}[9]{
\left(\begin{array}{c|cc|c|cc|}
	\multicolumn{1}{c}{\ddots} \\  \cline{2-4}
	& #1 & \multicolumn{1}{c}{#2} & #3 \\ 
	&      & \multicolumn{1}{c}{#4} & #5 \\ \cline{2-6}
	\multicolumn{1}{c}{}&&& \multicolumn{1}{c}{#1} &#6 &#7\\ 
	\multicolumn{1}{c}{}&&& \multicolumn{1}{c}{}&#8 &#9 \\ \cline{4-6}
	\multicolumn{1}{c}{}&&\multicolumn{1}{c}{}& \multicolumn{1}{c}{}&\multicolumn{1}{c}{} &\multicolumn{1}{c}{#1} 
\end{array}\right)}

\journal{Journal of XXXX}

\begin{document}

\begin{frontmatter}

\title{Discrete adjoint computations for relaxation Runge-Kutta methods}

\author{Mario J. Bencomo\corref{mycorrespondingauthor}}
\cortext[mycorrespondingauthor]{Corresponding author}
\ead{mjb6@rice.edu}

\author{Jesse Chan}
\address{Rice University, Houston TX}

\begin{abstract}

Relaxation Runge-Kutta methods reproduce a fully discrete dissipation (or conservation) of entropy for entropy stable semi-discretizations of nonlinear conservation laws. In this paper, we derive the discrete adjoint of relaxation Runge-Kutta schemes, which are applicable to discretize-then-optimize approaches for optimal control problems. Furthermore, we prove that the derived discrete relaxation Runge-Kutta adjoint preserves time-symmetry when applied to linear skew-symmetric systems of ODEs. Numerical experiments verify these theoretical results while demonstrating the importance of appropriately treating the relaxation parameter when computing the discrete adjoint.


\end{abstract}


\end{frontmatter}


\section{Introduction}

The \emph{relaxation Runge-Kutta} method was first introduced by \cite{Ketcheson_2019,Ranocha_2020a} for stability of time discretizations of ordinary differential equations (ODEs) with respect to a given inner-product norm, or in more general a convex entropy functional.
Recently, this relaxation approach has been generalized to multistep time integrators, \cite{Ranocha_2020b}.
We are interested in the application of these relaxation methods to optimal control problems, in hopes of leveraging stability properties especially for entropy stable semi-discretizations of nonlinear partial differential equations (PDEs).

In this paper we present novel adjoint computations and properties of the relaxation Runge-Kutta method.
The discrete linearization and adjoint of the relaxation Runge-Kutta method is derived by using a matrix representation, similar to \cite{Rothauge_2016}, and by using implicit differentiation in order to address the dependency of the relaxation parameter on the solution at current time steps.
Numerical experiments presented here highlight the importance of proper linearization.
We also prove, and demonstrate numerically, time-symmetry properties of the relaxation Runge-Kutta method when applied to linear skew-symmetric ODE systems, for general explicit and diagonally implicit Runge-Kutta methods.


Adjoint computations, by which we mean computations associated with the adjoint state method, are an efficient way to compute gradients of objective functionals in PDE-constrained optimization problems; see \cite{Gunzburger_2002,Antil_2018} for an overview of optimal control problems and the adjoint state method.
There are two main approaches associated with adjoint computations: one can either derive the adjoint state equations for the continuous problem and then discretize (referred to as the \textit{optimize-then-discretize} approach) or alternatively discretize the state equations first and then compute the adjoint (referred to as the \textit{discretize-then-optimize}).
The advent of \textit{automatic/algorithmic differentiation} (AD) has accelerated the advancement and utility of PDE-constrained optimization by essentially automating the adjoint state method and computation of sensitivities in a discretize-then-optimize approach, \cite{Griewank_2008,Griewank_2003}.
However, regardless of the convenience of AD, one must exercise caution since a discretize-then-optimize approach may produce a discretization that is inconsistent with the continuous optimization problem, e.g., \cite{Sirkes_1997}.

Issues with the discretize-then-optimize approach are dependent on the choice of discretization of the state equations, and much research has gone into analyzing this approach for different numerical schemes. 
Previous work on the discrete adjoint of Runge-Kutta methods showed that a discretize-then-optimize approach is indeed consistent, and moreover, that the adjoint of a Runge-Kutta method is yet another Runge-Kutta scheme of same order, \cite{Hager_2000,Walther_2007,Sandu_2006}.
In \cite{Serna_2016}, the author examines the links between symplectic Runge-Kutta methods and applications into the computation of sensitivities and adjoints.
See also \cite{Griewank_1993} for a more general paper on the discrete differentiation and convergence of iterative solvers.

The relaxation Runge-Kutta method can be viewed as an adaptive time step method, making the step-size dependent on the solution at previous time steps.
Previous work related to discrete adjoints of generic adaptive time stepping methods argues that taking the variable step-size into account in the linearization produces ``non-physical'' effects in the sensitivity and adjoint computations, \cite{Eberhard_1999,Alexe_2009}.
The authors also argue that the resulting discretize-then-optimize approach is inconsistent.
In this paper, however, the opposite is true, and a proper linearization of relaxation Runge-Kutta methods is not only consistent but also necessary for accuracy.


The paper is outlined as follows: 
In section~\ref{ssec:RK}, we present some notation, along with the standard Runge-Kutta method, as well as its discrete linearization and adjoint.
Next, in section~\ref{ssec:RRK}, we discuss the relaxation Runge-Kutta method, derive its discrete linearization and adjoint, and discuss its time-symmetry property.
In section~\ref{sec:Results}, the numerical experiments and results section, we demonstrate the importance of proper linearization for relaxation Runge-Kutta methods.
We also verify the time-symmetry property of these relaxation methods on a skew-symmetric linear problem.
Results are summarized in the conclusion, section~\ref{sec:Conclusion}.
Detailed derivations of discrete linearization and adjoint formulas are given in the appendix.

\section{Theory}\label{sec:theory}

Consider the following optimal control problem:
\begin{subequations}\label{eq:optcont}
\begin{align}
	\min_{\mb u} & \;\; \mathcal C(\mb y, \mb u)\\
	\text{s.t.} &\;\; \mathcal E(\mb y,\mb u) =  0 \label{eq:state}
\end{align}
\end{subequations}
where $\mb y$ and $\mb u$ denote the vectors of state and control variables respectively. 
The state equation \ref{eq:state} specifies a system of first-order initial value problem (IVP), potentially the semi-discretization of some PDE, of the form:
\begin{subequations}\label{eq:IVP}
\begin{align}
	\mb y'(t) &= \mb f(\mb y,\mb u, t), \quad 0 < t \le T\\
	\mb y(0) &= \mb y_{\rm init}(\mb u),
\end{align}
\end{subequations}
with $\mb y, \mb y_{\rm init}, \mb f(\mb y,\mb u,t)\in \mathbb R^N$.
Following a discretize-then-optimize approach, the continuous optimal control problem \ref{eq:optcont} is replaced by the following discrete optimization problem:
\begin{subequations}\label{eq:optdisc}
\begin{align}
	\min_{\fntb u} & \;\; \mathsf C(\fntb y, \fntb u) \\
	\text{s.t.} & \;\; \fntb E(\fntb y, \fntb u) = \fntb 0 \label{eq:stated}
\end{align}
\end{subequations}
where $\mathsf C$ and $\fntb E$ denote the discretized cost and state-equation operators respectively. 
Throughout this paper, we will use Sans Serif font to denote discretized quantities.
The equality constraint \ref{eq:stated} corresponds to the discretization of IVP \ref{eq:IVP} by some time stepping scheme, which in turn informs the discretization of the state and control vectors; we give more details in section \ref{ssec:RK} when discussing the Runge-Kutta method.

If the mapping $\fntb y \mapsto \fntb E(\fntb y,\fntb u)$ is invertible, then we can use the equality constraint \ref{eq:stated} to express the state variable as a function of the control variable.
In other words,
\[
	\fntb y = \fntb y(\fntb u) \DEF \fntb E^{-1}(\fntb 0,\fntb u).
\]
which allows us to reformulate \ref{eq:optdisc} as an uncontrained optimization problem with \textit{reduced} cost function
\[
	\wt{\mathsf C}(\fntb u) \DEF \mathsf C(\fntb y(\fntb u),\fntb u).
\]
Using implicit differentiation, one can show that the gradient of the reduced cost function is given by
\begin{equation}\label{eq:gradc}
	\nabla \wt{\mathsf C}(\fntb u) =
	\nabla_{\fntb u} \mathsf C(\fntb y, \fntb u)
	- \left( \frac{\partial \fntb E}{\partial \fntb u}(\fntb y,\fntb u) \right)^\top \fnts \lambda
\end{equation}
where $\fntb y$ must satisfy the state equation \ref{eq:stated}, while the adjoint-state (or co-state) vector $\fnts \lambda$ satisfies what is known as the \textit{adjoint equation},
\begin{equation} \label{eq:adjeq}
	\left( \frac{\partial \fntb E}{\partial \fntb y}(\fntb y,\fntb u) \right)^\top \fnts \lambda  = \nabla_{\fntb y} \mathsf C(\fntb y,\fntb u).
\end{equation}
Equations \ref{eq:gradc} and \ref{eq:adjeq} can also be derived from standard optimization theory via Lagrange multipliers.

Equations \ref{eq:gradc} and \ref{eq:adjeq} show that gradient computations of the cost function hinge on the linearization and adjoint of the state-equation operator, i.e., the choice of time integrator.
Before we derive the linearization and adjoint of the relaxation Runge-Kutta method, we present the well understood updates for standard Runge-Kutta.
For the majority of the paper, we drop the control vector $\fntb u$ and simply focus on the state-equation operator $\fntb E(\fntb y)$  (its linearization and adjoint) associated with discretizations of IVP \ref{eq:IVP}.

\subsection{Linearizations and adjoints of Runge-Kutta methods}
\label{ssec:RK}

A generic $s$-stage \textit{Runge-Kutta} (RK) method, specified by its coefficients
\[
	\fntb A_s \DEF \mat{a_{11}, & a_{12} & \cdots & a_{1s} \\
				      a_{21}, & a_{22} & \cdots & a_{2s} \\
				      \vdots  & \vdots   & \vdots & \vdots \\
				      a_{s1} & a_{s2} & \cdots & a_{ss}}, \quad 
	\fntb b_s \DEF \mat{b_1\\ b_2\\ \vdots \\ b_s}, \quad
	\fntb c_s \DEF \mat{c_1 \\ c_2\\ \vdots \\ c_s},
\]
applied to IVP \ref{eq:IVP}, yields the following time-stepping formulas:
\begin{subequations}\label{eq:RK}
\begin{align}
	\fntb y_{k} &= \fntb y_{k-1} + \Delta t \sum_{i=1}^s b_i \fntb F_{k,i}, \label{eq:RKstep}\\
	\fntb Y_{k,i} &= \fntb y_{k-1} + \Delta t \sum_{j=1}^{s}a_{ij} \fntb F_{k,j},\quad \text{for} \quad i=1,...,s, \label{eq:RKint}
\end{align}
\end{subequations}
where
\[
	\fntb F_{k,i} \DEF \mb f(\fntb Y_{k,i}, t_{k-1} + c_i\Delta t).
\]
and $\fntb y_k$ is an approximation to the solution at time $t_k = t_{k-1} + \Delta t$.
The $\fntb Y_{k,i}$ are referred to as the RK internal stages.
Assuming the RK method takes a total of $K$ steps, we define vector $\fnto y$ to be the concatenation of all $\fntb y_k$ and $\fntb Y_{k,i}$.
Let
\[
	\fntb Y_k \DEF \mat{\fntb Y_{k,1} \\ \vdots \\ \fntb Y_{k,s}} \in \mathbb R^{sN},
\]
then
\begin{equation}\label{eq:ybar}
	\fnto y \DEF \mat{\fntb y_0 \\ \fntb Y_1 \\ \fntb y_1 \\ \vdots \\ \fntb Y_K \\ \fntb y_K} \in \mathbb R^{\ol N}
\end{equation}
where $\ol N \DEF N + N(s+1)K$.
Throughout the remainder of this paper we will use an overline and bold Sans Serif font to denote vectors of dimension $\ol N$, whose components are grouped and denoted in a similar fashion to what was presented for $\fnto y$ in equation \ref{eq:ybar}.
For example, $\fnto w\in\mathbb R^{\ol N}$ is interpreted as
\[
	\fnto w = \mat{\fntb w_0 \\ \fntb W_1 \\ \fntb w_1 \\ \vdots \\ \fntb W_{K}\\ \fntb w_K}
\]
with $\fntb w_k\in \mathbb R^{N}$ and
\[
	\fntb W_k = \mat{\fntb W_{k,1} \\ \vdots \\ \fntb W_{k,s}} \in \mathbb R^{sN}.
\]

It can be shown that the discrete state equation \ref{eq:stated}, associated with an RK discretization, is of the form
\begin{equation}\label{eq:RKeq}
	\fntb E(\fnto y) \DEF \fntb L \fnto y - \fntb N(\fnto y) - \fnts \chi_0 \fntb y_{\rm init} = \fnto 0
\end{equation}
where $\fntb L$ is a lower unit-triangular matrix, and $\fntb N$ is a block lower-triangular (potentially nonlinear) operator acting on $\fnto y$ that involves the evaluation of the right-hand-side function $\mb f$ at internal stages.
The matrix $\fntb \chi_k\in \mathbb R^{\ol N\times N}$ is defined such that $\fnts \chi_k^\top \fnto y = \fntb y_k$.
Thus, for $\fntb v\in\mathbb R^N$, we have that $\fnto v = \fnts \chi_k \fntb v$ is a vector of length $\ol N$ with $\fntb v_k = \fntb v$ and zeros everywhere else. 
In other words, the term $\fnts \chi_0 \fntb y_{\rm init}$ accounts for the initial condition.
We refer to the appendix for more details concerning this ``matrix representation'' of the RK method, which borrows notation from \cite{Rothauge_2016}, and is used to derive the time-stepping formulas for the discrete linearization and adjoint of RK and its relaxation variant.
Note that the inverse mapping of $\fntb E$ is guaranteed to exist (barring any issues from the potentially nonlinear function $\mb f$) since $\fntb E$ will be a block lower-triangular operator.
In fact, time-stepping formulas \ref{eq:RK} specify this inverse map, solving the discrete state equation by forward substitution.

We provide the linearization of standard RK formulas \ref{eq:RK} as a point of reference for our discussion of adjoint computations for relaxation RK, which are derived by computing the Jacobian of the discrete state operator $\fntb E(\fnto y)$ in equation \ref{eq:RKeq} with respect to $\fnto y$.
The linearized RK time-stepping formulas are
\begin{subequations}\label{eq:linRK}
\begin{align}
	\fnts \delta_{k} &= \fnts \delta_{k-1} + \Delta t \sum_{i=1}^s b_i \fntb J_{k,i} \fnts \Delta_{k,i} + \fntb w_{k} \label{eq:linRKstep}\\
	\fnts \Delta_{k,i} &= \fnts \delta_{k-1} + \Delta t \sum_{j=1}^s a_{ij} \fntb J_{k,j} \fnts \Delta_{k,j} + \fntb W_{k,i}, \quad 1,...,s, \label{eq:linRKint}
\end{align}
\end{subequations}
where
\[
	\fntb J_{k,i} \DEF \frac{\partial \mb f}{\partial \mb y}(\fntb Y_{k,i}, t_{k-1}+c_i\Delta t).
\]
We refer to $\fnts \delta_k$ and $\fnts \Delta_k$ as the \emph{linearized RK approximation and linearized RK internal stages}, respectively.
These equations represent the solution to the following block lower-triangular system via forward substitution,
\[
	\fntb E'(\fnto y) \fnto \delta = \fnto w
\]
for some given right-hand-side vector $\fnto w\in\mathbb R^{\ol N}$, which can be used to incorporate initial conditions and/or source terms.
Note that  $\fntb J_{k,i}$ is the Jacobian matrix of $\mb f$ evaluated at the $i$-th internal stage $\fntb Y_{k,i}$.
In the case that $\mb f(\mb y(t),t)$ is linear in $\mb y$ we can expect $\fntb E'(\fnto y)\fnto y = \fntb E(\fnto y)$.

The adjoint RK formulas follow from solving
\[
	\fntb E'(\fnto y)^\top \fnto \lambda = \fnto w,
\] 
a block upper-triangular system, by back substitution:
\begin{subequations}\label{eq:adjRK}
\begin{align}
	\fnts \lambda_{k-1} &= \fnts \lambda_k + \sum_{i=1}^s \fnts \Lambda_{k,i} + \fntb w_{k-1}, \label{eq:adjRKstep}\\
	 \fnts \Lambda_{k,i} &= \Delta t \, \fntb J_{k,i}^\top \Big( b_i \fnts \lambda_{k} + \sum_{j=1}^s a_{ji} \fnts \Lambda_{k,j}\Big) + \fntb W_{k,i}, \quad i=1,...,s.\label{eq:adjRKint}
\end{align}
\end{subequations}
In the adjoint formulas \ref{eq:adjRK}, right-hand-side vector $\fnto w\in\mathbb R^{\ol N}$ incorporates source terms as well as a final time condition.
We refer to $\fnts \lambda_k$ and $\fnts \Lambda_k$ as the \textit{adjoint RK approximation} and \textit{adjoint RK internal stages}, respectively.
Note that the adjoint update formula uses $\fnts \lambda_{k}$ to update $\fnts \lambda_{k-1}$, in other words, we are marching backwards in time.
We also point out that this presentation of the adjoint RK method is based on the derivation from the matrix representation, unlike other presentations that reformulate the equations above to resemble an RK method; see \cite{Hager_2000,Sandu_2006}.

\subsection{Relaxation RK methods}
\label{ssec:RRK}

Let $\eta:\mathbb R^N \to \mathbb R$ denote the {\em entropy} function (smooth and convex) associated with IVP \ref{eq:IVP}, where the time evolution of the entropy is given by
\[
	\frac{d\eta}{dt}(\mb y(t)) = \nabla\eta(\mb y(t))^\top \, \mb f(\mb y(t),t).
\]
IVP \ref{eq:IVP} is said to be {\em entropy dissipative} if
\[
	\nabla\eta(\mb y(t))^\top \, \mb f(\mb y(t),t) \le 0
\]
or {\em entropy conservative} if 
\[
	\nabla \eta(\mb y(t))^\top\, \mb f(\mb y(t),t) = 0.
\]
In the discrete setting, we wish to ensure
\[
\begin{split}
	\eta(\fntb y_{k+1}) \le \eta(\fntb y_k),&\quad \text{(for entropy dissipative)},\\
	\text{or} \quad \eta(\fntb y_{k}) = \eta(\fntb y_0),& \quad \text{(for entropy conservative)}.
\end{split}
\]

The relaxation RK method achieves discrete entropy conservation/dissipation by modifying the update step as follows:
\begin{equation}\label{eq:RRKstep}
	\fntb y_{k} = \fntb y_{k-1} + \gamma_k \Delta t \sum_{i=1}^s b_i \fntb F_{k,i}
\end{equation}
where $\gamma_k$, the \emph{relaxation parameter}, is the non-zero root (closest to one) of the nonlinear scalar function $r_k(\,\cdot\,;\fnto y)$,
\begin{equation}\label{eq:root}
	r_k(\gamma;\fnto y) \DEF \eta\paren{\fntb y_{k-1} + \gamma \Delta t \sum_{i=1}^s b_i \fntb F_{k,i}} - \eta(\fntb y_{k-1}) - \gamma \Delta t \sum_{i=1}^s b_i \nabla\eta(\fntb Y_{k,i})^\top \fntb F_{k,i}.
\end{equation}
Internal stages $\fntb Y_{k,i}$ are calculated as in equation \ref{eq:RKint}.
After computing $\gamma_k$, there are two options for determining the solution at the next time step:
\begin{enumerate}
	\item[i.] {\em Incremental direction technique} (IDT) method: $\fntb y_{k} \approx \mb y(t_{k-1}+\Delta t)$
	\item[ii.] {\em Relaxation RK} (RRK) method: $\fntb y_{k} \approx \mb y(t_{k-1}+\gamma_k \Delta t)$
\end{enumerate}
Implementation wise, both IDT and RRK require the solution to a scalar root problem at each time step, though RRK resembles more of an adaptive time-stepping scheme given that $t_{k}= t_{k-1} + \gamma_k \Delta t$ at the moment the RRK approximation is updated.
In terms of accuracy, it was shown in \cite{Ranocha_2020a} that RRK methods preserve accuracy of the underlying RK scheme.
On the other hand, IDT schemes are order $p-1$ accurate for an underlying method of order $p$.
For the purposes of a clean presentation, we first derive the linearized and adjoint formulas for the simpler IDT method, and then extend to RRK with some minor modifications.

\subsubsection{Discrete linearization and adjoint of IDT}

The state-equation operator, $\fntb E$, associated with IDT is modified in accordance to equation \ref{eq:RRKstep} by adding a dependency on the relaxation parameter vector $\fnts \gamma \DEF (\gamma_1,\dots, \gamma_{K})^\top\in\mathbb R^{K}$.
Assuming IDT takes $K$ steps, where $K$ is the number of steps taken by the underlying RK method, we have
\[
	\fntb E(\fnto y,\fnts \gamma) \DEF \fntb L \fnto y - \fntb N(\fnto y,\fnts \gamma) - \fnts\chi_0 \fntb y_{\rm init}.
\]
The relaxation parameter $\gamma_k$ is defined by the solution to a root equation at each time step.
This root equation can be written in vector form as
\[
	\fntb r(\fnts \gamma; \fnto y) 
	\DEF \mat{r_1(\gamma_1; \fnto y)\\ r_2(\gamma_2; \fnto y)\\ \vdots}
	= \mat{0 \\ 0 \\ \vdots}.
\]
Note that the equation above implicitly defines the relaxation parameter vector as a function of $\fnto y$, i.e., $\fnts \gamma = \fnts \gamma(\fnto y)$.
Using this, we denote the \textit{reduced} state-equation operator by 
\[
	\wt{\fntb E}(\fnto y)\DEF \fntb E(\fnto y, \fnts\gamma(\fnto y)).
\]

A proper linearization of the IDT method will require the Jacobian of the reduced state-equation operator,
\[
	\wt{\fntb E}'(\fnto y) = 
	\frac{\partial \fntb E}{\partial \fnto y}(\fnto y, \fnts\gamma(\fnto y)) 
	+ \frac{\partial \fntb E}{\partial \fnts \gamma}(\fnto y,\fnts\gamma(\fnto y))\, \fnts\gamma'(\fnto y).
\]
Unfortunately, one cannot directly compute $\fnts \gamma'$ since the relaxation parameters are computed numerically using some iterative root-solving algorithm.
One could bypass this issue by simply taking $\partial\fntb E/\partial\fnto y$ as the Jacobian, ignoring the second term above, essentially viewing the relaxation parameter as a constant in the linearization, as suggested by \cite{Eberhard_1999,Alexe_2009}.
This approach would result in linearized and adjoint time-stepping formulas that are almost identical to RK, equations \ref{eq:linRKstep} and \ref{eq:adjRKstep} (and equation \ref{eq:adjRKint} for the adjoint internal stages), but with weights $\fntb b_s \mapsto \gamma_k\fntb b_s$.
We will show in our numerical results that ignoring the relaxation parameter will have negative consequences.

For a proper linearization, we compute $\fnts \gamma'$ via implicit differentiation.
Note that $\gamma_k$ is dependent on $\fntb y_{k-1}$ and $\fntb Y_k$ only, thus it suffices to compute partial derivatives with respect to these variables.
In particular,
\begin{subequations}\label{eq:gradgamma}
\begin{align}
	(\nabla_y \gamma_k)^\top \DEF 
	\frac{\partial \gamma_k}{\partial \fntb y_{k-1}}(\fnto y)
		&= -\paren{\frac{\partial r_k}{\partial \gamma}}^{-1} \frac{\partial r_k}{\partial\fntb y_{k-1}}\eval{(\fnts\gamma(\fnto y);\fnto y)}{} \in \mathbb R^{1\times N}\\
	(\nabla_Y \gamma_{k,i})^\top \DEF 
	\frac{\partial \gamma_k}{\partial \fntb Y_{k,i}}(\fnto y)
		&= -\paren{\frac{\partial r_k}{\partial \gamma}}^{-1} \frac{\partial r_k}{\partial\fntb Y_{k,i}}\eval{(\fnts\gamma(\fnto y);\fnto y)}{} \in \mathbb R^{1\times N},\\
	(\nabla_Y \gamma_{k})^\top &= \Big( (\nabla_Y \gamma_{k,1})^\top, ..., (\nabla_Y \gamma_{k,s})^\top \Big) \in \mathbb R^{1\times sN},
\end{align}
\end{subequations}
where $r_k$, again, is defined in \ref{eq:root} and
\begin{subequations}\label{eq:dr}
\begin{align}
	\frac{\partial r_k}{\partial \gamma}(\fnts\gamma(\fnto y);\fnto y)
		&= \Delta t \sum_{i=1}^s b_i \Big( \nabla\eta(\fntb y_{k}) - \nabla\eta(\fntb Y_{k,i}) \Big)^\top \fntb F_{k,i}\\
	\frac{\partial r_k}{\partial \fntb y_{k-1}}(\fnts\gamma(\fnto y);\fnto y)
		&= \nabla\eta(\fntb y_{k})^\top - \nabla\eta(\fntb y_{k-1})^\top\\
	\frac{\partial r_k}{\partial \fntb Y_{k,j}}(\fnts\gamma(\fnto y);\fnto y)
		&= \gamma_k b_j \Delta t \left\{ \Big( \nabla\eta(\fntb y_{k})-\nabla\eta(\fntb Y_{k,j}) \Big)^\top \fntb J_{k,j} - \fntb F_{k,j}^\top \nabla^2\eta(\fntb Y_{k,j}) \right\}, \quad j=1,...,s.
\end{align}
\end{subequations}
We provide the remaining details in the appendix, and simply state the resulting time-stepping formulas in the following lemmas.

\begin{lemma}\label{thm:linIDT}
	The linearized IDT time-stepping formulas are
	\begin{equation}\label{eq:linIDTstep}
		\fnts \delta_{k} = \fnts \delta_{k-1} + \gamma_k \Delta t \sum_{i=1}^s b_i \fntb J_{k,i} \fnts \Delta_{k,i}
			\red{\,+\, \rho_k\; \Delta t \sum_{i=1}^s b_i \fntb F_{k,i}} + \fntb w_{k}
	\end{equation}
	where
	\[
		\red{\rho_k =  (\nabla_{y}\gamma_k)^\top \fnts \delta_{k-1} + (\nabla_{Y}\gamma_k)^\top \fnts \Delta_k}
	\]
	and the computation of the internal stages $\fnts \Delta_k$ is the same as for RK (see equation \ref{eq:linRKint}).
\end{lemma}

\begin{lemma}\label{thm:adjIDT}
	The adjoint IDT time-stepping formulas are
	\begin{equation}\label{eq:adjIDTstep}
		\fnts \lambda_{k-1} = \fnts \lambda_{k} + \sum_{i=1}^s \fnts \Lambda_{k,i} + \red{\xi_{k}\nabla_{y}\gamma_k} +\fntb w_{k-1}
	\end{equation}
	where
	\begin{align}\label{eq:adjIDTint}
			\red{ \xi_{k} } & \red{ = \Delta t \sum_{i=1}^s b_i \fntb F_{k,i}^\top \fnts \lambda_{k} } \nonumber \\ 
			 \fnts \Lambda_{k,i} &= \Delta t\; \fntb J_{k,i}^\top \Big( \gamma_k b_i \fnts \lambda_{k} + \sum_{j=1}^s a_{ji} \fnts \Lambda_{k,j}\Big) \red{\,+\, \xi_{k}  \nabla_{Y}\gamma_{k,i} } + \fntb W_{k,i}.
	\end{align}
\end{lemma}

As we show in the appendix section, linearized and adjoint IDT formulas define solutions to the following matrix systems, respectively:
\begin{align*}
	\paren{ \frac{\partial\fntb E}{\partial\fnto y} + \red{\frac{\partial\fntb E}{\partial \fntb\gamma} \fntb \gamma'} } \fnto \delta &= \fnto w, \quad \text{(linearized)}\\
	\paren{ \frac{\partial \fntb E}{\partial\fnto y}^\top + \red{ \paren{\fntb \gamma'}^\top \frac{\partial\fntb E}{\partial \fntb\gamma}^\top }} \fnto \lambda &= \fnto w, \quad \text{(adjoint)}.
\end{align*}
Terms associated with the linearization of $\gamma_k$ in the linearized/adjoint IDT update formulas above, and in Lemmas \ref{thm:linIDT} and \ref{thm:adjIDT}, are highlighted using red text.
We see that taking into account the relaxation parameter in the linearization results in having to compute gradients $\nabla_{y}\gamma_k$ and $\nabla_{Y}\gamma_k$ as well as scalars $\rho_k$ and $\xi_{k}$ at each time step.
These gradients in turn require RK approximations at two time steps ($\fntb y_{k-1}$ and $\fntb y_{k}$), internal stages $\fntb Y_k$, and the evaluation of the right-hand-side function $\mb f$ and its Jacobian, as well as the gradient and Hessian of the entropy function.

\subsubsection{Time-symmetry of IDT}
Before moving on to the discrete RRK adjoint, we discuss the special case where
\[
	\mb f(\mb y(t) , t) = \fntb S\mb y(t)
\]
for a skew-symmetric matrix $\fntb S\in\mathbb R^{N\times N}$ in IVP \ref{eq:IVP}.
It can be shown that IVP \ref{eq:IVP} is entropy/energy conservative with respect to the square entropy
\[
	\eta(\mb y(t)) = \frac{1}{2}\|\mb y(t)\|^2.
\]
Of course, given that $\mb f(\mb y)$ is linear in this case, it follows that linearization of the continuous state equation coincides with the original problem, ignoring initial conditions.
This is not obvious but still true for the IDT formulas, even though the relaxation parameter is nonlinear with respect to $\fnto y$.
 In other words, discretization by an IDT method commutes with linearization in this special case.

\begin{theorem}\label{thm:IDTskew1}
Suppose we were to apply IDT and the linearized IDT to IVP \ref{eq:IVP} with $\mb f(\mb y(t), t) = \fntb S \mb y(t)$, where $\fntb S$ is skew-symmetric and the entropy function is $\eta(\mb y) = \frac{1}{2}\|\mb y\|^2$.
It follows that IDT updates are equivalent to linearized IDT updates.
In particular, if linearized IDT is applied with $\fnts \delta_0= \fntb y_{\rm init}$, $\fntb w_k=\fntb 0$ and $\fntb W_{k}=\fntb 0$, then $\fntb \delta_k = \fntb y_k$ for all time steps.
\end{theorem}

\begin{proof}
Before proving equivalence between IDT and linearized IDT, we first make some observations.
In this linear case, we have
\[
	\fntb F_{k,i} = \fntb S\fntb Y_{k,i}, \quad \text{and} \quad 
	\fntb J_{k,i} = \fntb S.
\]
Moreover, the root function for computing the relaxation parameter is quadratic in $\gamma$,
\[
	r_k(\gamma; \fnto y) = \frac{1}{2}\| \fntb y_{k-1} + \gamma \fntb d_k\|^2 - \frac{1}{2}\|\fntb y_{k-1}\|^2 - \gamma \fntb e_k
\]
where
\begin{align}
	\fntb d_k &:= \Delta t \sum_{i=1}^s b_i \fntb F_{k,i} = \Delta t \sum_{i=1}^s b_i \fntb S \fntb Y_{k,i}, \label{eq:dk}\\
	\fntb e_k &:= \Delta t \sum_{i=1}^s b_i \nabla\eta(\fntb Y_{k,i})^\top \fntb F_{k,i}. \nonumber
\end{align}
The vectors $\fntb d_k$ and $\fntb e_k$ are based on notation from \cite{Ranocha_2020a} and represent the search direction (based on a projection interpretation of relaxation methods) and the entropy production at the current time step, respectively.
Note that $\fntb e_k = 0$ since $\nabla \eta(\mb y) = \mb y$ and
\[
	\nabla \eta(\fntb Y_{k,i})^\top \fntb F_{k,i} = \fntb Y_{k,i}^\top \fntb S \fntb Y_{k,i} = 0
\]
by skew-symmetry of $\fntb S$.

Since $r_k(\gamma_k;\fnto y) = 0$ is nothing more than a quadratic root problem we can come up with an explicit formula for $\gamma_k$, the non-zero root of this quadratic function:
\begin{equation}\label{eq:gammak}
	\gamma_k := -\frac{2\fntb y_{k-1}^\top \fntb d_k}{\|\fntb d_k\|^2},
\end{equation}
which is consistent with what is reported in \cite{Ketcheson_2019}.
Furthermore, the gradients with respect to $\fntb y_{k-1}$ and $\fntb Y_{k,i}$ are given by
\begin{align}
	\nabla_{y} \gamma_k &= -\frac{2}{\|\fntb d_k\|^2} \fntb d_k, \label{eq:grady}\\
	\nabla_{Y} \gamma_{k,i} 
		&= -\frac{2 b_i \Delta t}{\|\fntb d_k\|^2} \fntb S^\top \Big( \underbrace{\fntb y_{k-1} - \frac{2 \fntb y_{k-1}^\top \fntb d_k}{\|\fntb d_k\|^2} \fntb d_k}_{\fntb y_{k} = \fntb y_{k-1} + \gamma_k\fntb d_k} \Big)
		= -\frac{2 b_i \Delta t}{\|\fntb d_k\|^2} \fntb S^\top \fntb y_{k}. \label{eq:gradY}
\end{align}

We now prove that the $k$-th step of the linearized IDT method is equivalent to the $k$-th IDT step, assuming $\fnts \delta_{k-1}= \fntb y_{k-1}$.
First note that the internal stages coincide, $\fnts \Delta_{k}= \fntb Y_{k}$, since their update formulas are identical, see equations \ref{eq:RKint} and \ref{eq:linRKint}. This follows from $\fnts \delta_{k-1}= \fntb y_{k-1}$, $\fntb J_{k,j} = \fntb S$ and $\fntb F_{k,j} = \fntb S \fntb Y_{k,j}$.
Next we look at the update step for linearized IDT, equation \ref{eq:linIDTstep}.
In particular,
\begin{align*}
	\nabla_{y} \gamma_k^\top \fnts \delta_{k-1} = -\frac{2}{\|\fntb d_k\|^2} \fntb d_k^\top \fntb y_{k-1} = \gamma_k
\end{align*}
and
\begin{align*}
	\nabla_{Y} \gamma_{k}^\top \fnts \Delta_k 
	&= - \frac{2}{\|\fntb d_k\|^2} \fntb y_{k}^\top \paren{ \Delta t \sum_{i=1}^s b_i \fntb S\fntb Y_{k,i} } \\
	&= - \frac{2}{\|\fntb d_k\|^2} \fntb y_{k}^\top\fntb d_k \\
	&= - \frac{2}{\|\fntb d_k\|^2} (\fntb y_{k-1} + \gamma_k \fntb d_k)^\top\fntb d_k\\
	&= -\gamma_k,
\end{align*}
which implies that 
\[
	\rho _k = \nabla_{y} \gamma_k^\top \fnts \delta_{k-1} + \nabla_{Y} \gamma_k^\top \fnts \Delta_k = 0.
\]
Using $\fnts \delta_{k-1}=\fntb y_{k-1}$, $\fnts \Delta_{k,i} = \fntb Y_{k,i}$, and $\rho_k = 0$, we can conclude that the $k$-th step is given by
\begin{align*}
	\fnts \delta_k 
	&= \fnts \delta_{k-1} + \gamma_k \underbrace{\Delta t \sum_{i=1}^s b_i \fntb S \fnts \Delta_{k,i}}_{\fntb d_k} + \cancel{\rho_k} \Delta t \sum_{i=1}^s b_i \fntb S\fntb Y_{k,i}\\
	&= \fntb y_{k-1} + \gamma_k \fntb d_k\\
	&= \fntb y_k.
\end{align*}
Assuming $\fnts \delta_0 = \fntb y_0 = \fntb y_{\rm init}$, we can conclude by induction that $\fnts \delta_k = \fntb y_k$ and $\fnts \Delta_{k}=\fntb Y_{k}$ for all time steps.
\end{proof}

Building off of the equivalence of the forward and linearized continuous systems, and similarly for IDT and linearized IDT, we discuss an interesting relationship with their respective adjoints.
The adjoint of IVP \ref{eq:IVP}, again with $\mb f(\mb y,t) = \fntb S\mb y$, is given by
\begin{subequations}\label{eq:IVPadj}
\begin{align}
	-\mb\lambda'(t) &= \fntb S^{\top} \mb\lambda(t), \quad 0<t<T \label{eq:ODEadj}\\
	\mb\lambda(T) &= \fnts\lambda_{\rm final}
\end{align}
\end{subequations}
along with the following adjoint condition
\begin{equation}\label{eq:adjcond}
	\fnts\lambda_{\rm final}^\top \mb y(T) = \mb\lambda(0)^\top \fntb y_{\rm init}.
\end{equation}
For skew-symmetric $\fntb S$, system \ref{eq:IVPadj} is essentially the original problem but with a final time condition.
In particular, if $\fntb\lambda_{\rm final} = \mb y(T)$, then system \ref{eq:IVPadj} is equivalent to the forward problem in reverse time, i.e., $\mb \lambda(t) = \mb y(t)$ for all time $0\le t\le T$.
Note that condition \ref{eq:adjcond} is automatically satisfied here since the problem is norm conservative,
\[
	\fnts\lambda_{\rm final}^\top \mb y(T) = \|\mb y(T)\|^2 = \| \mb y(0)\|^2 = \mb\lambda(0)^\top \fntb y_{\rm init}.
\]
We refer to this equivalency between the forward and adjoint systems as a \textit{time-symmetry} property of the continuous problem. 
In Theorem~\ref{thm:IDTskew2} we prove that IDT preserves a similar time-symmetry with its adjoint.

\begin{theorem}\label{thm:IDTskew2}
Assume the underlying RK method is explicit or diagonally implicit.
Suppose we were to apply IDT to IVP \ref{eq:IVP} with $\mb f(\mb y, t) = \fntb S \mb y$, where $\fntb S$ is skew-symmetric and the entropy function is $\eta(\mb y) = \frac{1}{2}\|\mb y\|^2$.
It follows that the IDT method preserves time-symmetry.
In particular, if adjoint IDT is applied with $\fnts \lambda_{K}= \fntb y_K$ (assuming IDT takes $K$ steps), $\fntb w_k = \fntb 0$ and $\fntb W_k=\fntb 0$, then $\fnts \lambda_k = \fntb y_{k}$ for all time steps, up to machine precision.
\end{theorem}

\begin{proof}
This proof makes use of some simplifications derived in the first half of the proof for Theorem~\ref{thm:IDTskew1}, based on the linearity of the problem.
In particular, we make use of $\fntb F_{k,i} = \fntb S \fntb Y_{k,i}$, $\fntb J_{k,i} = \fntb S$, $\fntb d_k$ as defined by equation \ref{eq:dk}, the explicit formula for $\gamma_k$ given by equation \ref{eq:gammak}, and gradients of $\gamma_k$ (equations \ref{eq:grady} and \ref{eq:gradY}).

Consider the $k$-th step (in reverse time) of the adjoint IDT algorithm, assuming $\fnts \lambda_{k}=\fntb y_k$.
Note that
\[
	\xi_k = \Delta t \sum_{i=1}^s b_i \fntb F_{k,i}^\top \fnts \lambda_k = \fntb d_k^\top \fntb y_k.
\]
The internal stages are given by
\begin{align*}
	\fnts \Lambda_{k,i} 
	&= \Delta t \fntb J_{k,i}^\top \paren{ \gamma_{k} b_i \fnts \lambda_k + \sum_{j=1}^s a_{ji} \fnts \Lambda_{k,j} } + \xi_{k} \Big(\nabla_{Y} \gamma_{k,i}\Big)\\
	&= \Delta t \fntb S^\top \paren{ \gamma_{k} b_i \fntb y_k + \sum_{j=1}^s a_{ji} \fnts \Lambda_{k,j} } + \paren{ \fntb d_{k}^\top \fntb y_k } \paren{ -\frac{2 \Delta t \; b_i }{\|\fntb d_{k}\|^2} \fntb S^\top \fntb y_{k} }\\
	&= \Delta t \fntb S^\top \paren{\sum_{j=1}^s a_{ji} \fnts \Lambda_{k,j} }  + \gamma_{k} b_i \Delta t  \fntb S^\top\fntb y_k  - \gamma_{k} b_i \Delta t \fntb S^T \fntb y_k\\
	&=  \Delta t \fntb S^\top \paren{ \sum_{j=1}^s a_{ji} \fnts \Lambda_{k,j} },
\end{align*}
where we made use of
\begin{align*}
	- \frac{2}{\|\fntb d_k\|^2} \fntb d_{k}^\top\fntb y_k 
	&= - \frac{2}{\|\fntb d_k\|^2} \fntb d_k^\top (\fntb y_{k-1} + \gamma_k \fntb d_k)\\
	&= \gamma_k - 2\gamma_k\\
	&= -\gamma_k,
\end{align*}

Recall that the RK coefficient matrix $\fntb A_s$ is lower triangular for explicit or diagonally implicit RK method.
Thus, we have the following set of equations for the internal stages,
\[
	\paren{\fntb I - a_{ii}\Delta t \fntb S^\top} \fnts \Lambda_{k,i} = \Delta t \fntb S^\top \sum_{j>i} a_{ji} \fnts \Lambda_{k,j}, \quad i=1,...,s.
\]
It is easy to see, that for $i=s$,
\[
	\paren{\fntb I - a_{ss}\Delta t \fntb S^\top} \fnts \Lambda_{k,s} = \fntb 0 \quad \implies \quad \fnts \Lambda_{k,s} = \fntb 0.
\]
Moreover, by induction, it can be shown that $\fnts \Lambda_{k,i}=\fntb 0$ for all $i=1,...,s$.

Since the internal stages zero-out, the $k$-th adjoint update step reduces to
\begin{align*}
	\fnts \lambda_{k-1} 
	&= \fnts \lambda_{k} + \sum_{i=1}^s \cancel{\fnts \Lambda_{k,i}} + \xi_{k}\nabla_{y} \gamma_{k}\\
	&= \fnts \lambda_{k} + \paren{ \fntb d_{k}^\top \fntb y_k } \paren{ -\frac{2}{\|\fntb d_{k}\|^2} \fntb d_{k} }\\
	&= \fntb y_{k} - \gamma_{k} \fntb d_{k}\\
	&= \fntb y_{k-1}.
\end{align*}
Assuming $\fnts \lambda_K = \fntb y_K$, we can conclude by induction that $\fnts \lambda_{k} = \fntb y_{k}$ for all subsequent time steps.
\end{proof}

Again, we emphasize that our notion of \emph{adjoint} in this paper is based the transpose of the matrix-form of the linearized state-equations induced by the time-stepping scheme.
Moreover, by \textit{time-symmetry} we refer to the relationship between the forward and adjoint continuous system as well as the forward and adjoint IDT time-stepping formulas.
In other words, for IDT, time-symmetry refers to the adjoint time-step as reversing the corresponding forward time-step.
The use of \text{adjoint} and \text{time-symmetry} should not be equated with the standard definitions used in the literature for symplectic/geometric numerical integrators; see \cite{Hairer_2006,Hernandez_2018} for other definitions of time-reversibility, time-symmetry, and adjoints of time-stepping methods.

\subsubsection{Discrete linearization and adjoint of RRK}

Recall that RRK yields the update $\fntb y_{k}\approx \mb y(t_{k-1}+\gamma_k\Delta t)$.
For this reason RRK can be interpreted as an adaptive time-stepping scheme.
Consequently, special care is taken in order to ensure that at the end of the time loop we end up with an approximation at the desired final time.
Suppose that at time step $K-1$ we have $t_{K-1}+\Delta t>T$ and $ t_{K} = t_{K-1} + \gamma_{K}\Delta t > T$.
In other words, the RRK time loop terminates after $K$ steps.
One could apply some continuation method to interpolate an approximation at $t=T$.
This interpolation step will need to be entropy stable in some sense and be accounted for in the linearization and subsequently adjoint of the RRK method.

As an interpolation-free alternative, we take an IDT final step but with a corrected step size of $\Delta t^* \DEF T-t_{K-1}$, resulting in $t_{K}=T$.
In other words, 
\begin{align*}
	\fntb y_K &= \fntb y_{K-1} + \gamma_K\Delta t^* \sum_{i=1}^s b_i \fntb F_{K,i},\\
	\fntb Y_{K,i} &= \fntb y_{K-1} + \Delta t^* \sum_{j=1}^s a_{ij} \fntb F_{K,j}, \quad i=1,...,s,
\end{align*}
where $\fntb y_K \approx \mb y(T)$.
This approach is similar to what is considered in \cite{Eberhard_1999,Alexe_2009} for generic adaptive time-stepping methods.
Note that 
\begin{equation}\label{eq:dt*}
	\Delta t^* = T - t_0 - \Delta t\sum_{\ell=1}^{K-1} \gamma_\ell,
\end{equation}
which implies that $\Delta t^*$ is dependent on $\gamma_\ell$ (and thus $\fntb y_{\ell-1}$ and $\fntb Y_\ell$) for $\ell=1,...,K-1$.
Moreover, since $\Delta t^*$ is the step size used in $r_K$ this implies $\gamma_K$ is dependent on $(\fntb y_{\ell-1},\fntb Y_\ell)$ for $\ell=1,...,K$.
These dependencies must be taken into account for a proper linearization.
Again, we provide the derivation on the appendix and summarize the results here.

\begin{lemma}\label{thm:linRRK}
Assuming RRK takes $K$ steps, and that the last step is taken as an IDT step, with $\Delta t^*=T-t_{K-1}$, then the linearized RRK approximations and internal stages, $\fnts \delta_k$ and $\fnts \Delta_k$ for $k=1,...,K-1$, are as given by the linearized IDT computations in equations \ref{eq:linIDTstep} and \ref{eq:linRKint}, respectively.
The linearized RRK update formula for the final step is given by
\[
	\fnts \delta_{K} = \fnts \delta_{K-1} + \gamma_{K}\Delta t^* \sum_{i=1}^s b_i \fntb J_{K,i} \fnts \Delta_{K,i} + \red{\rho_{K} \Delta t^* \sum_{i=1}^s b_i \fntb F_{K,i}} + \fntb w_K
\]
where
\[
\begin{split}
	 \fnts \Delta_{K,i} 
	 	&= \fnts \delta_{K-1} + \Delta t^* \sum_{j=1}^s a_{ij}  \fntb J_{K,j} \fnts \Delta_{K,j} 
		\;\blue{- \; \rho_* \Delta t \sum_{j=1}^s a_{ij} \fntb F_{K,j}} + \fntb W_{K,i},\\
 	\blue{\rho_*} &\blue{= \sum_{k=1}^{K-1}\rho_k}.
\end{split}
\]
\end{lemma}

\begin{lemma}\label{thm:adjRRK}
Assuming RRK takes $K$ steps, and that the last step is taken as an IDT step with $\Delta t^*=T-t_{K-1}$, then the adjoint RRK approximations and internal stages for the first adjoint step, $\fnts \lambda_{K-1}$ and $\fnts \Lambda_{K}$, are given by the adjoint IDT formulas, equations \ref{eq:adjIDTstep} and \ref{eq:adjIDTint} respectively, though with $\Delta t\mapsto \Delta t^*$.
The adjoint RRK update formulas for $k=K-1,...,1$ are given by
\[
	\fnts \lambda_{k-1} = \fnts \lambda_{k} + \sum_{i=1}^s \fnts \Lambda_{k,i} + (\red{\xi_{k}} \; \blue{-\; \xi_*}) \red{\nabla_{y}\gamma_k} + \fntb w_{k-1}
\]
where
\begin{align*}
	\fnts \Lambda_{k,i} &= \Delta t \; \fntb J_{k,i}^\top \Big( \gamma_{k} b_i \fnts \lambda_{k} + \sum_{j=1}^s a_{ji} \fnts \Lambda_{k,j}\Big) + (\red{\xi_{k}} \;\blue{-\;\xi_*}) \red{ \nabla_{Y}\gamma_{k,i}} + \fntb W_{k,i},\\
	\blue{\xi_*} &\blue{= \Delta t \sum_{i=1}^s \sum_{j=1}^s a_{ji}\fntb F_{K,i}^\top\fnts \Lambda_{K,j}}.
\end{align*}
\end{lemma}

In relation to the matrix-form, linearized and adjoint IDT formulas define solutions to the following matrix systems, respectively:
\begin{align*}
	\paren{ \frac{\partial\fntb E}{\partial\fnto y} + \red{\frac{\partial\fntb E}{\partial \fntb\gamma} \fntb \gamma'} \blue{+ \frac{\partial \fntb E}{\partial \Delta t^*} \frac{d \Delta t^*}{d \fnto y} }
	} \fnto \delta &= \fnto w, \quad \text{(linearized)}\\
	\paren{ \frac{\partial \fntb E}{\partial\fnto y}^\top + \red{\paren{\fntb \gamma'}^\top \frac{\partial\fntb E}{\partial \fntb\gamma}^\top} \blue{+\paren{\frac{d\Delta t^*}{d \fnto y}}^\top \frac{\partial\fntb E}{\partial\Delta t^*}^\top } } \fnto \lambda &= \fnto w, \quad \text{(adjoint)}.
\end{align*}
Again, terms associated with linearization with respect to the relaxation parameter are in red. 
In blue we have terms related to linearization with respect to the final step size $\Delta t^*$.
We note that accounting for $\Delta t^*$ in the linearization requires computing the scalar quantity $\rho_*$, which can be done efficiently by simply accumulating the $\rho_k$ scalars.
This scalar shows up in the internal stage computations for the last time-step.
For the adjoint RRK method, we see an expected structure that is complementary to linearized RRK.
In particular, the final step is as given by the adjoint IDT formula (with $\Delta t \mapsto \Delta t^*$). 
The scalar $\xi_*$, computed from the final step, shows up as a correction term for the remaining time steps.

Consider again the special case when $\mb f(\mb y,t) = \fntb S \mb y$, where $\fntb S$ is skew-symmetric.
The same results we observed for IDT hold for RRK as well.

\begin{theorem}\label{thm:RRKskew1}
Suppose we were to apply RRK, and linearized RRK to IVP \ref{eq:IVP} with $\mb f(\mb y, t) = \fntb S \mb y$, where $\fntb S$ is skew-symmetric.
It follows that RRK updates are equivalent to linearized RRK updates.
In particular, if linearized RRK is applied with $\fntb \delta_0= \fntb y_{\rm init}$, $\fntb w_k=\fntb 0$ and $\fntb W_k=\fntb 0$, then $\fntb \delta_k = \fntb y_k$ for all time steps.
\end{theorem}

\begin{proof}
Given that the first $K-1$ steps of RRK are algorithmically identical to IDT, it follows from theorem~\ref{thm:linIDT} that $(\fnts \delta_{k-1}, \fnts \Delta_k) = (\fntb y_{k-1}, \fntb Y_k)$ for $k=1,...,K-1$.
Moreover, in the proof of theorem~\ref{thm:linIDT}, we showed that $\rho_k = 0$ for $k=1,...,K-1$ ($\rho_K=0$ can similarly be proven), thus the accumulated scalar $\rho_*$ in RRK is also zero.
From this it is easy to see that the proposition holds.
\end{proof}

\begin{theorem}\label{thm:RRKskew2}
Assume the underlying RK method is explicit or diagonally implicit.
Suppose we were to apply RRK to IVP \ref{eq:IVP} with $\mb f(\mb y, t) = \fntb S \mb y$, where $\fntb S$ is skew-symmetric.
It follows that the RRK method preserves time-symmetry.
In particular, if adjoint RRK is applied with $\fnts \lambda_{K}= \fntb y_K$ (assuming RRK takes $K$ steps), $\fntb w_k=\fntb 0$ and $\fntb W_k = \fntb 0$, then $\fnts \lambda_k = \fntb y_{k}$ for all time steps, up to machine precision.
\end{theorem}

\begin{proof}
In adjoint RRK, we have to account for the auxiliary scalar $\xi_*$,
\begin{align*}
	\xi_* = \Delta t \sum_{ij} a_{ji} \fntb F_{K,i}^\top \fnts \Lambda_{K,j}. 
\end{align*}
Just as in the IDT case, the first set of internal stages, $\fnts \Lambda_{K,i}$ for $i=1,...,s$, are all zero, hence $\xi_*=0$.
This allows us to carry on and show that $\fnts \lambda_{k-1} = \fntb y_{k-1}$ for all time steps.
\end{proof}

\section{Numerical experiments and results}
\label{sec:Results}

This section presents numerical results that validate and highlight properties discussed in the previous sections.
We restrict our focus to the more interesting and applicable RRK method, though analogous results can be generated for IDT as well.
We mention in passing that a bisection algorithm, similar to that of \cite{Ranocha_2019_code}, is used to solve the scalar root subproblem when computing relaxation parameters at each time step.

The previous section presented the discrete linearization and adjoint of RRK, taking into consideration the relaxation parameter, $\gamma$, and the corrected final step-size, $\Delta t^*$.
We demonstrate various consequences of improper linearization, when $\fnts \gamma$ and $\Delta t^*$ are not considered in the linearization; we refer to these cases as the $\gamma$-constant or $\Delta t^*$-constant case respectively.
Recall that $\Delta t^*$ can be written in terms of $\gamma_\ell$ for $\ell=1,...,K-1$ (see equation \ref{eq:dt*}),
hence, in the $\gamma$-constant case $\Delta t^*$ is also considered constant.

\subsection{Mathematical models}

We consider three mathematical models throughout our numerical tests which we present at this moment.

\subsubsection*{Nonlinear pendulum}

For the nonlinear pendulum, we take the first-order form as done in \cite{Ranocha_2020a},
\begin{equation}\label{eq:nlinpen}
	\frac{d}{dt} \mat{y_1(t)\\ y_2(t)} = \mat{-\sin(y_2(t)) \\ y_1(t)}, \quad 0 < t \le T
\end{equation}
with initial condition
\[
	\mb y(0) = \fntb y_{\rm init} \DEF \mat{1.5 \\ 1}.
\]
This problem is entropy conservative with respect to the following entropy function,
\[
	\eta(\mb y) = \frac{1}{2} y_1^2 - \cos(y_2).
\]

\subsubsection*{1D compressible Euler}
The 1D compressible Euler equations are given by
\begin{subequations}\label{eq:Euler}
\begin{align}
	\frac{\partial \rho}{\partial t} + \frac{\partial(\rho u)}{\partial x} &= 0,\\
	\frac{\partial(\rho u)}{\partial t} + \frac{\partial(\rho u^2 + p)}{\partial x} &= 0,\\
	\frac{\partial E}{\partial t} + \frac{\partial(u(E+p))}{\partial x} &=0,
\end{align}
\end{subequations}
which we solve over $0<t\le 1.5$ and $-1\le x\le 1$, with initial condition
\begin{align*}
	\rho(x,0) &= 1+ \tfrac{1}{2} \exp( - 50 (x-0.1)^2),\\
	u(x,0) &= 0,\\
	 p(x,0) &= \rho(x,0)^\gamma,
\end{align*}
assuming the pressure, $p$, satisfies the following constitutive relation:
\[
	p = (\gamma-1)(E - \tfrac{1}{2}\rho u^2), \quad \gamma=1.4.
\]
We arrive at IVP \ref{eq:IVP} by discretizing in space the compressible Euler equations with an entropy stable DG scheme \cite{Chan_2018,Chan_2020}.
In particular, $\mb y(t)$ represents semi-discretized quantities related to the variables $\rho,\rho u$ and $E$.
The entropy function associated with this semi-discretized system corresponds to a discretization of the continuous total entropy,
\[
	\eta(\mb y(t)) \approx \int S(\rho,\rho u,E) \; dx
\]
where $S$ is the continuous entropy function
\[
	S(\rho,\rho u,E) \DEF -\frac{\rho s}{\gamma-1}, \quad s = \log\paren{\frac{p}{\rho^\gamma}}.
\]

\subsubsection*{Linear skew-symmetric system}

For experiments concerning time-symmetry of RRK, as detailed in theorem \ref{thm:RRKskew2}, we solve the following linear problem:
\begin{subequations}\label{eq:skewsym}
\begin{align}
	\mb y'(t) &= \fntb S \mb y(t), \quad 0<t<T,\\
	\mb y(0) &= \fntb y_{\rm init}
\end{align}
\end{subequations}
where $\fntb S\in\mathbb R^{N\times N}$ is a randomly generated skew-symmetric matrix and $\fntb y_{\rm init}\in\mathbb R^N$ is a randomly generated initial condition; we take $N=10$.
The final time is taken to be proportional to the Frobenius norm of $\fntb S$, specifically, 
\[
	T = 10\cdot \|\fntb S\|_F \approx 133.46.
\]
As previously mentioned, this system is entropy conservative with respect to 
\[
	\eta(\mb y(t)) = \frac{1}{2} \| \mb y(t)\|^2.
\] 

\subsection{RK schemes}

The following are the RK schemes used throughout our experiments:

\begin{itemize}
\item 
Heun's method, a 2-stage, second-order RK scheme which we refer to as RK2,
\[
 	\fntb A_2 = \mat{0 & 0\\ 1 & 0}, \quad
	\fntb b_2 = \mat{1/2\\ 1/2}, \quad
	\fntb c_2 = \mat{0\\1}.
\]

\item
A 3-stage, third-order RK scheme (see \cite{Shu_1988}), which we refer to as RK3,
\[
 	\fntb A_3 = \mat{0 & 0 & 0\\ 1 & 0 & 0\\ 1/4 & 1/4 & 0}, \quad
	\fntb b_3 = \mat{1/6\\ 1/6 \\ 2/3}, \quad
	\fntb c_3 = \mat{0\\1\\1/2}.
\]

\item 
The standard 4-stage, fourth-order RK scheme, referred to as RK4,
\[
	\fntb A_4 = \mat{ 0 & 0 & 0 & 0\\ 1/2 & 0 & 0 & 0\\ 0 & 1/2 & 0 & 0\\ 0 & 0 & 1& 0}, \quad
	\fntb b_4 = \mat{1/6 \\ 1/3 \\ 1/3 \\ 1/6}, \quad
	\fntb c_4 = \mat{0 \\ 1/2 \\ 1/2 \\ 1}.
\]

\item
A  3-stage, third-order DIRK scheme,  referred to as DIRK3, with
\[
	\fntb A_3 = \mat{ \alpha & 0 & 0 \\ \tau_2-\alpha & \alpha & 0 \\ b_1 & b_2 & \alpha}, \quad
	\fntb b_3 = \mat{ b_1 \\ b_2 \\ \alpha}, \quad
	\fntb c_3 = \mat{\alpha \\ \tau_2 \\ 1},
\]
where
\begin{align*}
	\alpha &= 0.435866521508459\\
	\tau_2 &= (1+\alpha)/2\\
	b_1 &= -(6\alpha^2 -16\alpha +1)/4\\
	b_2 &= (6\alpha^2 - 20\alpha + 5)/4.
\end{align*}
We refer the reader to \cite{Persson_2012,Alexander_1977} for additional details on this DIRK scheme.
\end{itemize}

All relaxation variants of a specified RK scheme are simply denoted by an extra ``R'' in front of RK.
For example, RRK4 and DIRRK3 refer to the relaxation variants of RK4 and DIRK3, respectively.

\subsection{Verifying linearizations}

The first set of numerical experiments verify our linearization RRK formulas, as well as highlight the importance of taking into consideration the relaxation parameter and the corrected final step-size in the linearization process.
Let
\[
	\fntb E(\fnto y) = \fnto w
\]
denote an abstract system of time-stepping equations for a given right-hand-side vector $\fnto w$.
Let $\fntb H$ denote the inverse of $\fntb E$, i.e., if $\fnto y$ is the solution to the equation above then $\fnto y=\fntb H(\fnto w)$.
Essentially, $\fntb H$ specifies the explicit update formulas of a given time-integration scheme.
From implicit differentiation, it follows that the directional derivative of $\fntb H$, evaluated at $\fnto w$ in direction $\fnto d$, is given by $\fnto \delta$, 
\[
	\fnto \delta = \fntb H'(\fnto w) \fnto d.
\]
In practice, we compute $\fnto \delta$ as the solution to
\[
	 \fntb E'(\fnto y) \fnto \delta = \fnto d.
\]
This motivates our work in computing the Jacobian $\fntb E'$ for the different schemes. 

To verify that our linearized code indeed computes derivatives we study the numerical convergence of a simple finite difference approximation to the directional derivative, similar to what is done in \cite{Wilcox_2015} under a different context.
In particular, we check that
\begin{equation}\label{eq:fderr}
	\left\| \frac{\fntb H(\fnto w + h \fnto d) - \fntb H(\fnto w) }{h} - \fntb H'(\fnto w)\fnto d  \right\|
	=
	\left\| \frac{\fnto y_h - \fnto y}{h} - \fnto \delta \right\|
	= \mathcal O(h)
\end{equation}
as $h \to 0^+$, where $\fnto y$, $\fnto y_h$ and $\fnto\delta$ are solutions to the following systems:
\begin{align*}
	\fntb E(\fnto y) &= \fnto w,\\
	\fntb E(\fnto y_h) &= \fnto w + h\fnto d,\\
	\fntb E'(\fnto y) \fnto \delta &= \fnto d.
\end{align*}

\subsubsection{Nonlinear pendulum results}

We use the nonlinear pendulum equations \ref{eq:nlinpen} as our first model for tests concerning proper linearization.
When computing the finite difference error in equation \ref{eq:fderr}, we take $\fnto w = \fnto 0$ and $\fnto d = \fnts \chi_0 \fntb d_0$ where $\fntb d_0$ is a random vector meant to represent a random initial condition.

Figure \ref{fig:nlpen_linerr} plots finite difference error \ref{eq:fderr} for different choices of Jacobian $\fntb E'$ representing linearized RRK formulas with proper linearization or for the $\gamma$-constant or $\Delta t^*$-constant cases.
As expected, we observe that proper linearization achieves linear convergence while the other two cases do not, which is apparent for lower order RK schemes.
Interestingly enough, the $\Delta t^*$-constant case yields significantly smaller errors than the $\fntb\gamma$-constant case, though both fail to converge.

\begin{figure}[!h]
	\centering
	\begin{subfigure}{0.45\textwidth}
		\centering
		\includegraphics[width=\textwidth]{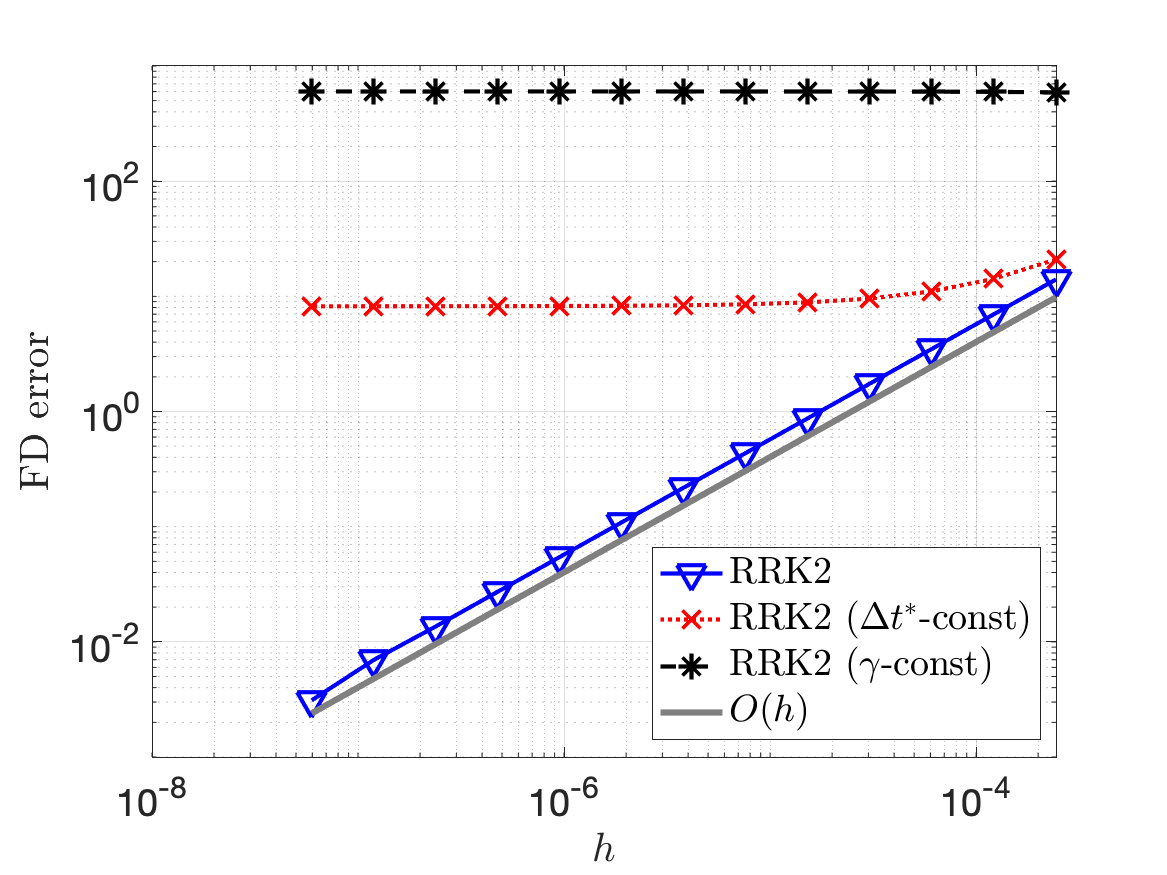}
		\caption{Using RK2.}
		\label{fig:nlpen_linerr_RRK2}
	\end{subfigure}
	\hspace{1em}
	\begin{subfigure}{0.45\textwidth}
		\centering
		\includegraphics[width=\textwidth]{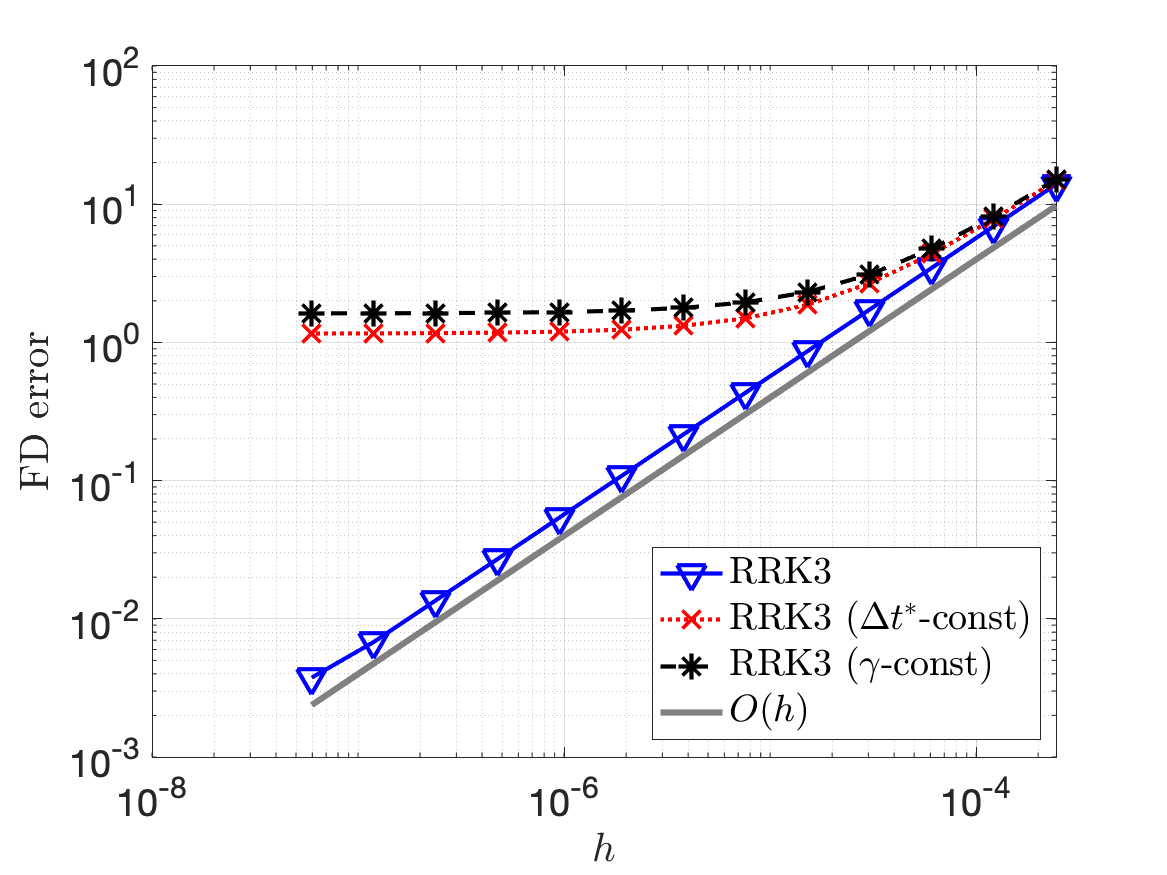}
		\caption{Using RK3.}
		\label{fig:nlpen_linerr_RRK3}
	\end{subfigure}
	\begin{subfigure}{0.45\textwidth}
		\centering
		\includegraphics[width=\textwidth]{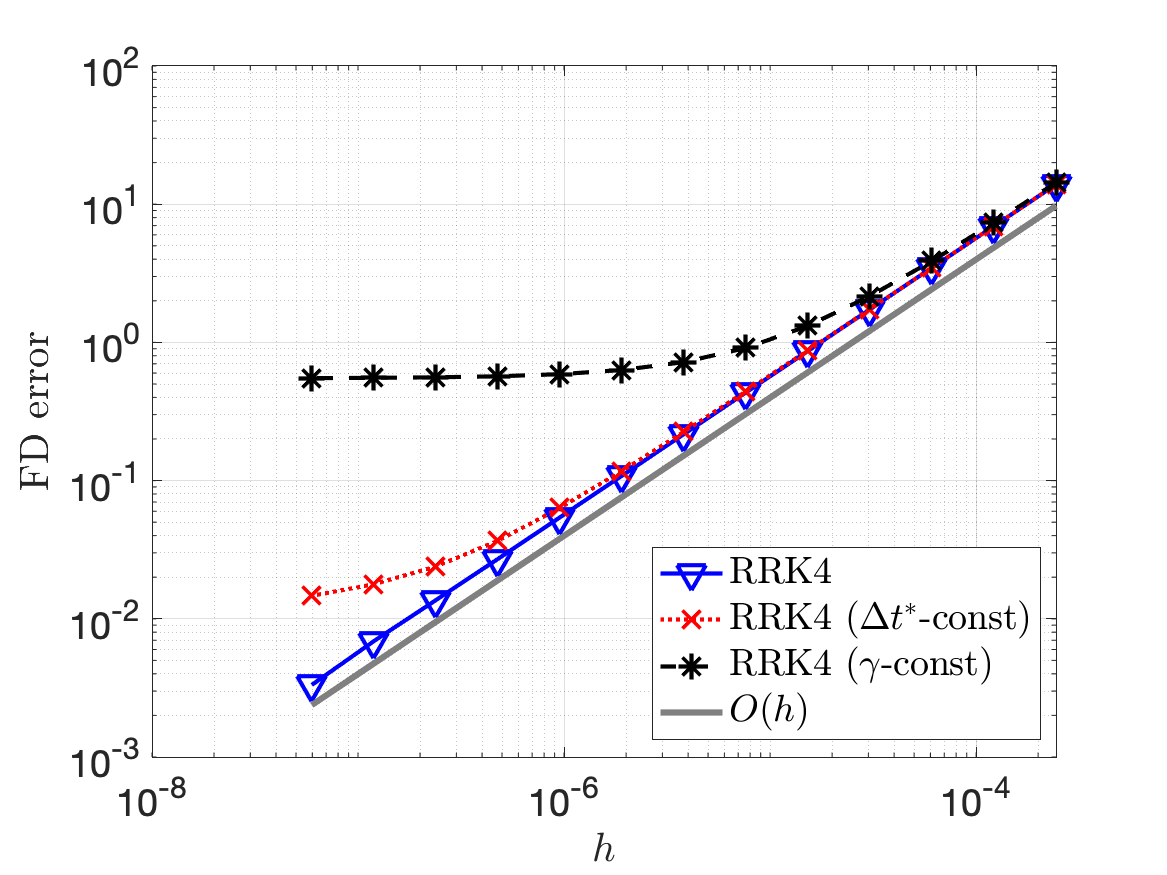}
		\caption{Using RK4.}
		\label{fig:nlpen_linerr_RRK4}
	\end{subfigure}
	\caption{Convergence of the finite difference error \ref{eq:fderr} for RRK with nonlinear pendulum model; $\Delta t = 0.1$ and $T=200$.}
	\label{fig:nlpen_linerr}
\end{figure}

\subsubsection{1D Compressible Euler equations results}

We run the same FD convergence test with proper and improper linearizations of RRK when solving the semi-discretization of the one-dimensional compressible Euler equations \ref{eq:Euler}.
For the step size, we use 
\[
	\Delta t = \text{CFL} \times \frac{h}{C_N}, \quad C_N = \frac{(N+1)^2}{2}
\]
where $h=1/16$ is the size of the DG element, $N=3$ the polynomial order of the DG method, and $\text{CFL}$ is a user-defined constant.
	
Figure~\ref{fig:euler_linerr} demonstrates the FD convergence, or lack of, for different choice of $\text{CFL}$ constant.
As before, linear convergence is achieved clearly under proper linearization.
Larger step sizes reveal that indeed the improper linearizations fail to converge.
Unlike the nonlinear pendulum example, however, the $\Delta t^*$-constant case yielded larger errors than the $\gamma$-constant case.

\begin{figure}[!h]
	\centering
	\begin{subfigure}{0.45\textwidth}
		\centering
		\includegraphics[width=\textwidth]{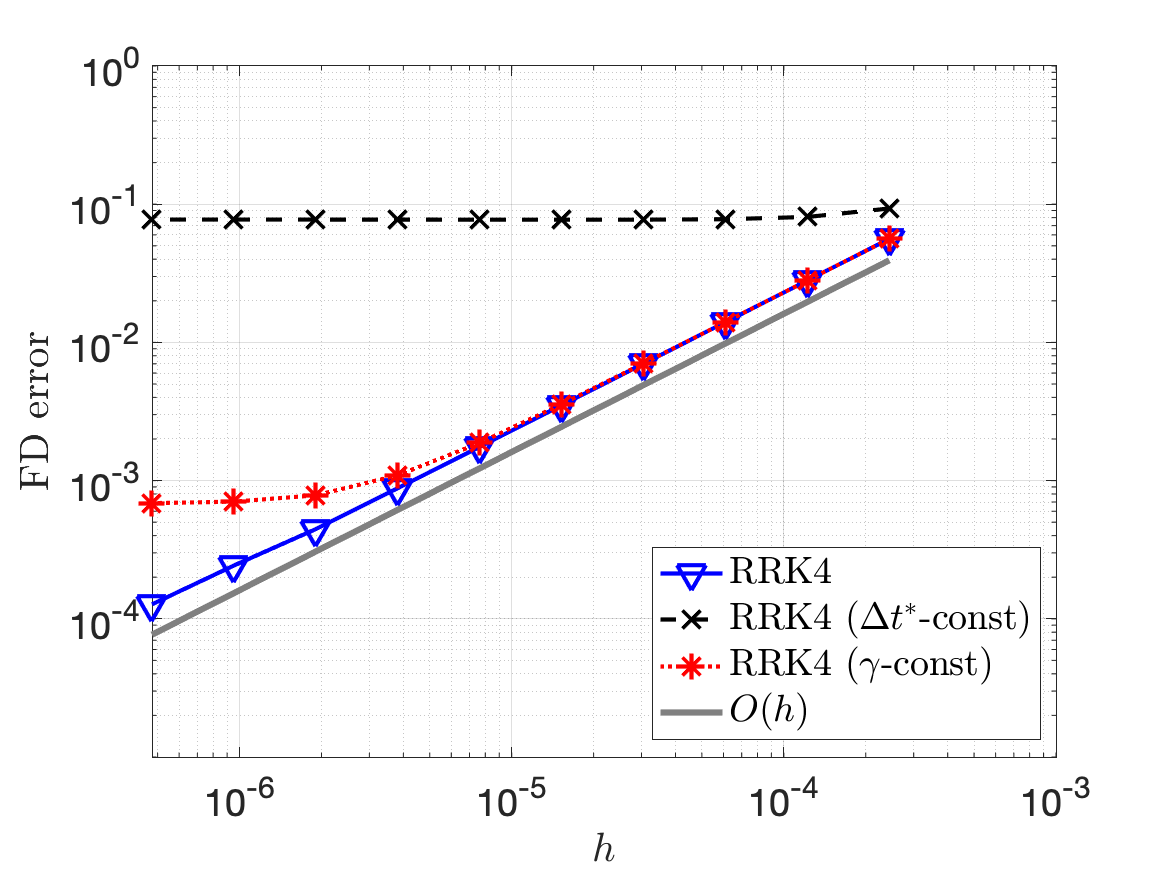}
		\caption{Using RK4, $\text{CFL}=1$.}
		\label{fig:euler_linerr_CFL1}
	\end{subfigure}
	\hspace{1em}
	\begin{subfigure}{0.45\textwidth}
		\centering
		\includegraphics[width=\textwidth]{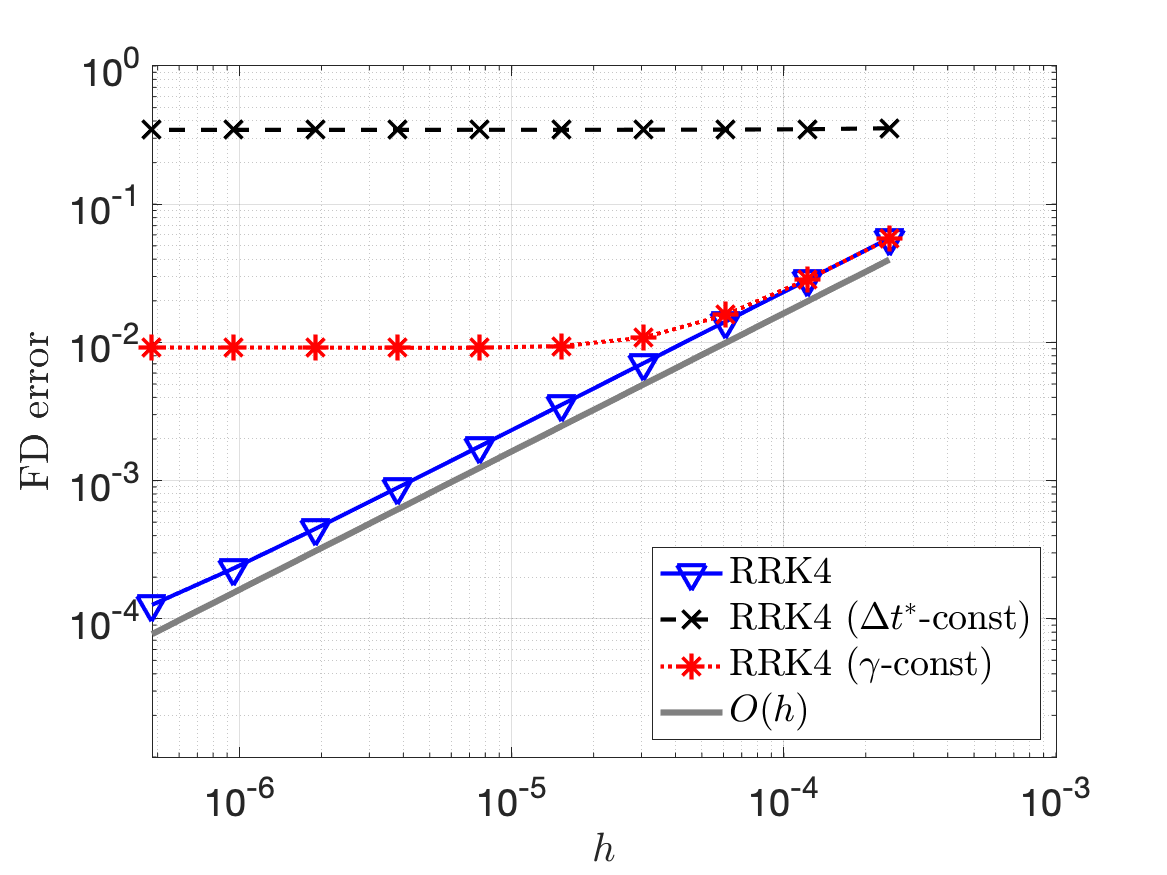}
		\caption{Using RK4, $\text{CFL}=1.5$.}
		\label{fig:euler_linerr_CFL2}
	\end{subfigure}
	\caption{Convergence of the finite difference error \ref{eq:fderr} for RRK with 1D compressible Euler model.}
	\label{fig:euler_linerr}
\end{figure}

\subsection{Consistency of discrete RRK adjoints}

As mentioned in the introduction, the discretize-then-optimize approach yields a discrete adjoint equation \ref{eq:adjeq} for gradient computations of cost functions.
An optimize-then-discretize approach would have resulted in a continuous analog, with some continuous adjoint equation.
The discrete adjoint equation is said to be \textit{consistent} if its solution converges to the solution of the continuous adjoint equation.

Depending on the choice of time-integrator, the resulting adjoint equation may or may not be consistent.
For RK time-integrators, it has been shown that the discrete adjoint equation correspond to an RK discretization (of same order) of the continuous adjoint equation, \cite{Sandu_2006}.
Concerning adaptive time-integrators, \cite{Eberhard_1999,Alexe_2009} argue that taking into consideration the adaptive step-size in the linearization process can result in inconsistent adjoint scheme.
We present here a convergence study of the discrete RRK adjoint to address these concerns and verify (at least numerically) consistency.

Consider control problem \ref{eq:optcont} with cost function
\begin{align*}
	\mathcal C(\fntb u) =  \frac{1}{2}\| \mb y(T)\|^2,
\end{align*}
subject to
\begin{align*}
	\frac{d}{dt} \mat{y_1(t)\\ y_2(t)} &= \mat{-\sin(y_2(t))\\ y_1(t)}, \quad 0<t\le T,\\
	\mb y(0) &= \fntb u,
\end{align*}
where the initial condition is the control variable.
The discretized optimal control problem is given by \ref{eq:optdisc} with
\begin{align*}
	\mathsf C(\fnto y,\fntb u) 
		&= \frac{1}{2} \| \fntb y_K\|^2 = \frac{1}{2} \fnto y^\top \fnts\chi_K \fnts\chi_K^\top \fnto y,\\
	\fntb E(\fnto y,\fntb u) 
		&= \fntb L - \fntb N(\fnto y) - \fnts \chi_0 \fntb u,
\end{align*}
assuming the state equation has been discretized by an RK/RRK scheme, taking $K$ steps to arrive at the final time.
The gradient of the reduced, discrete cost function is
\[
	\nabla \wt{\mathsf C}(\fntb u) = \fnts \chi_0^\top \fnto \lambda = \fnts \lambda_0
\]
where $\fnto\lambda$ is the solution to the discrete adjoint equation,
\begin{equation}\label{eq:adjeq2}
	\paren{ \frac{\partial \fntb E}{\partial \fnto y}(\fnto y,\fntb u)}^\top \fnto \lambda = \fnts\chi_K\fnts\chi_K^\top \fnto y.
\end{equation}
Note that $\fnto y$ in the discrete adjoint problem corresponds to the solution of the forward problem, that is, the state equation $\fntb E(\fnto y,\fntb u) = \fnto 0$.
Moreover, the right-hand-side of the adjoint equation here simply dictates the final time condition, specified by $\fntb y_K$.

In the numerical experiments presented here, we examine the convergence of 
\[
	\text{forward error} = \frac{\|\fntb y_K - \mb y_{\rm ref}(T)\|}{\|\mb y_{\rm ref}(T)\|}, 
	\quad
	\text{adjoint error} = \frac{\|\fntb \lambda_0 -  \lambda_{\rm ref}(0)\|}{\| \lambda_{\rm ref}(0)\|}
\]
as $\Delta t\to 0$, where $\mb y_{\rm ref}$ and $\lambda_{\rm ref}$ are computed using RK4 (and its adjoint) with a small step-size of $\Delta t=10^{-5}$.
Given that RK4 and its adjoint are consistent discretizations of the continuous state and adjoint state equations, we can expect that our reference solution will be sufficiently accurate for these tests.

Figure~\ref{fig:nlpen_conv_RRK_fwd} shows the well expected convergence of RK and RRK schemes of different order for the forward problem.
In the adjoint problem, figure~\ref{fig:nlpen_conv_RRK_adj}, we observe that adjoint RRK, along with its improper linearization variants, is indeed consistent. 
Moreover, optimal convergence rates are achieved by adjoint RRK with proper linearization and the $\gamma$-constant case.
Interestingly enough, we see that we lose an order of convergence in the $\Delta t^*$-constant case.

\begin{figure}[!h]
	\centering
	\includegraphics[width=0.7\textwidth]{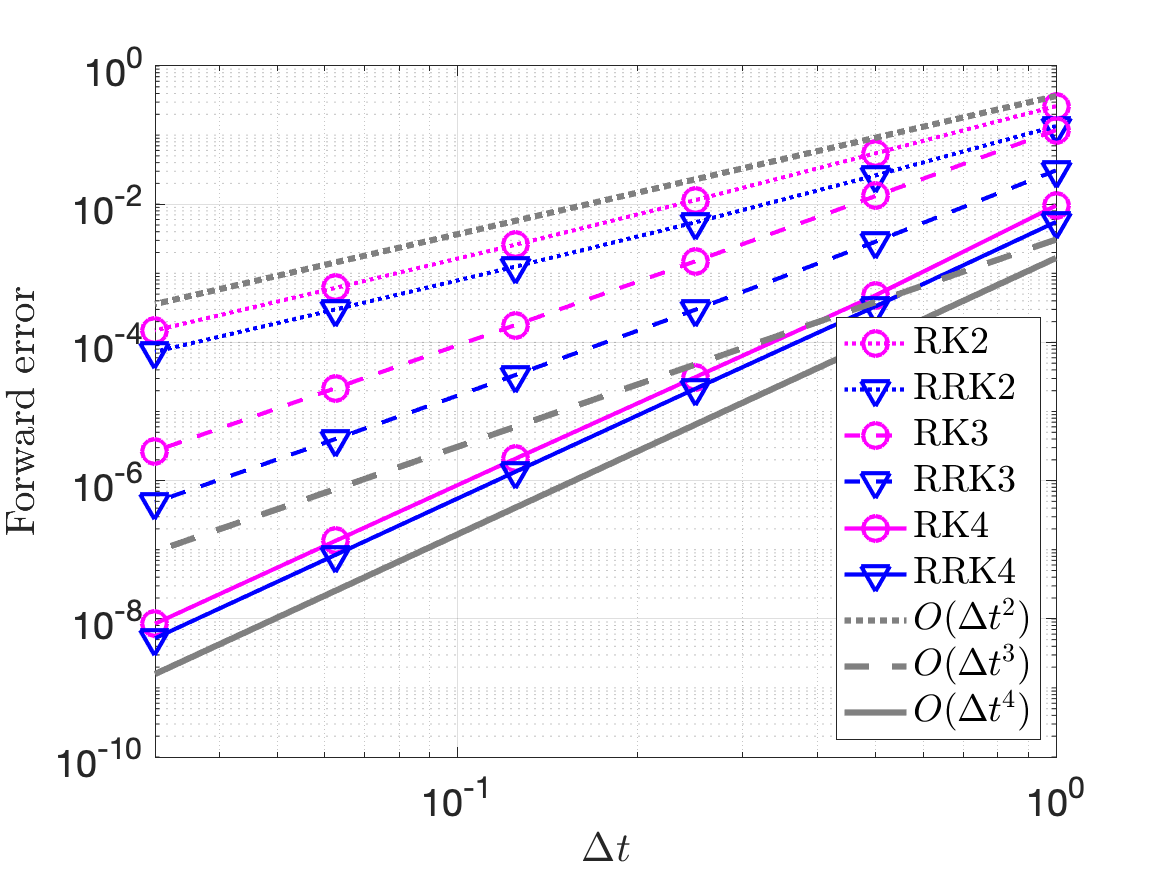}
	\caption{Convergence of RK and RRK discretizations of nonlinear pendulum problem; $T=2$.}
	\label{fig:nlpen_conv_RRK_fwd}
\end{figure}

\begin{figure}[!h]
	\centering	
	\begin{subfigure}{0.45\textwidth}
		\centering
		\includegraphics[width=\textwidth]{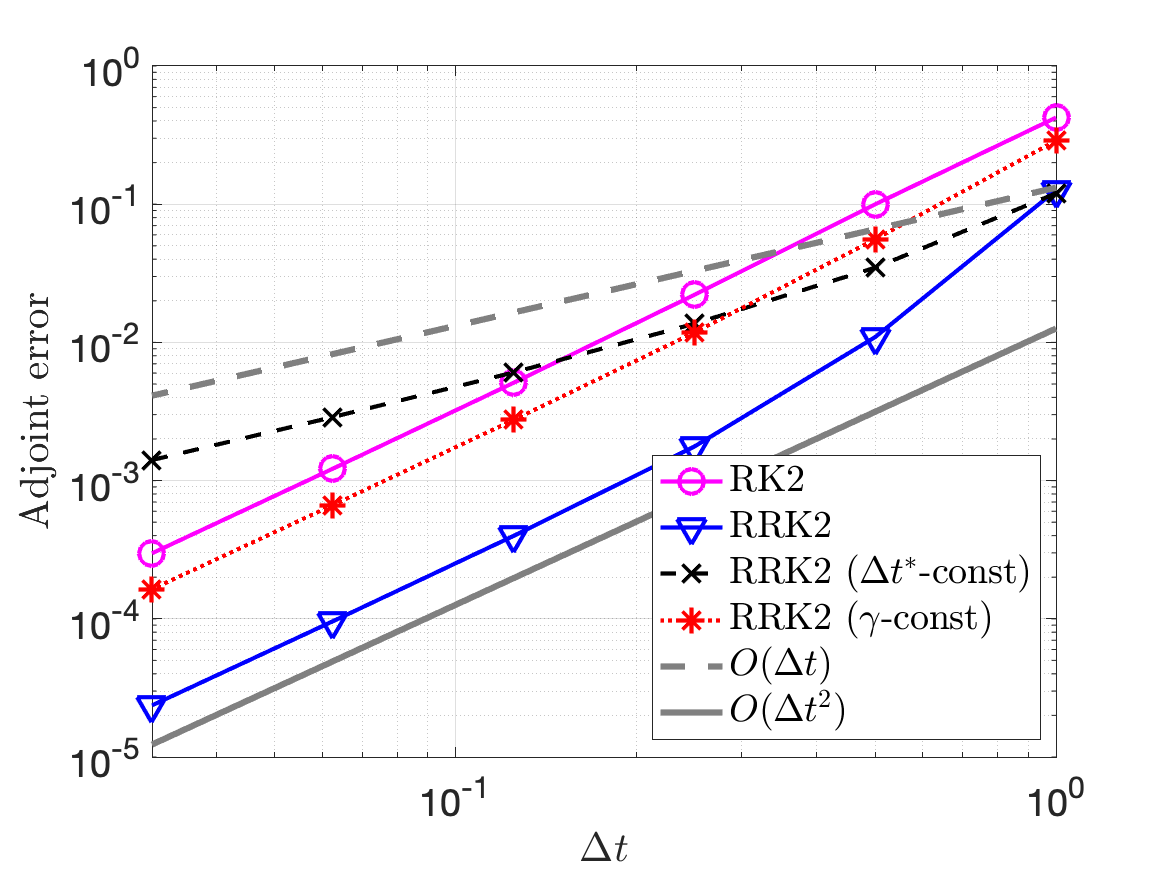}
		\caption{Using RK2.}
		\label{fig:nlpen_conv_RRK2_adj}
	\end{subfigure}
	\hspace{1em}
	\begin{subfigure}{0.45\textwidth}
		\centering
		\includegraphics[width=\textwidth]{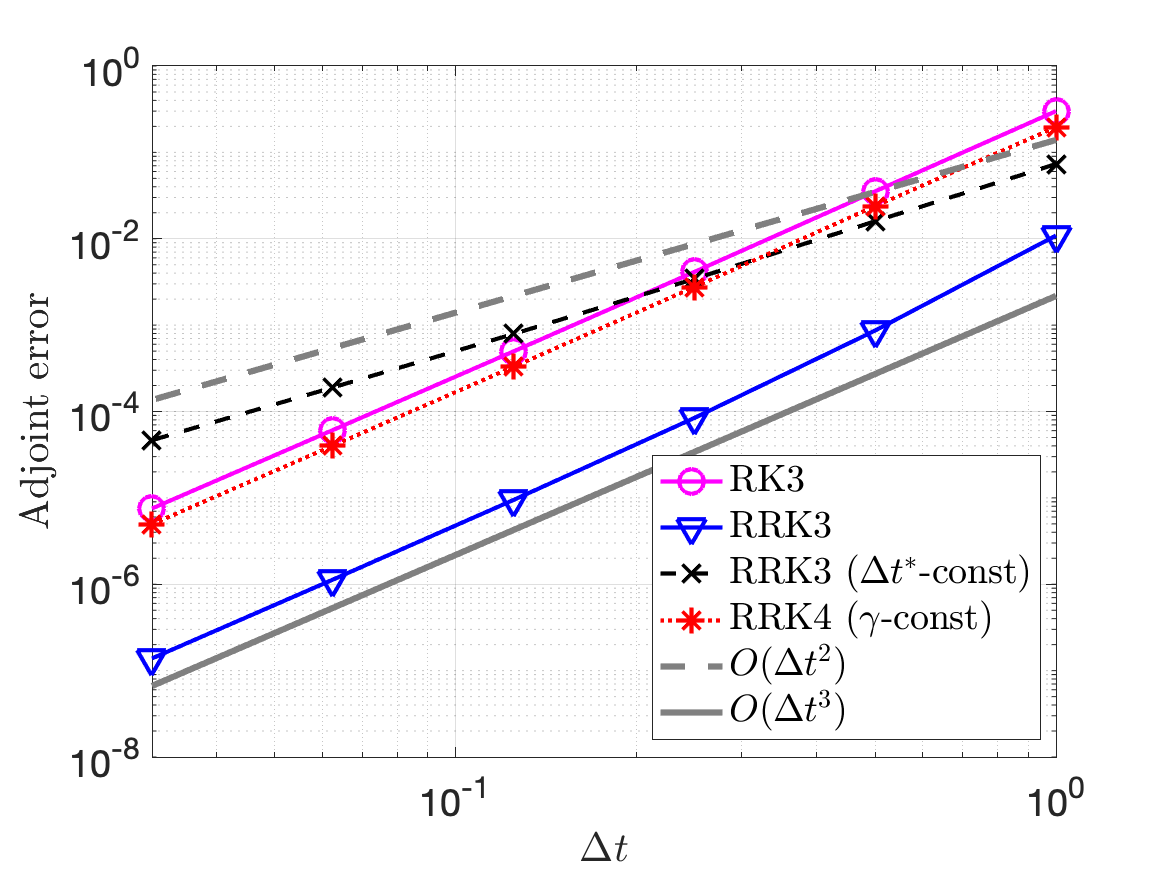}
		\caption{Using RK3.}
		\label{fig:nlpen_conv_RRK3_adj}
	\end{subfigure}
	\\
	\begin{subfigure}{0.45\textwidth}
		\centering
		\includegraphics[width=\textwidth]{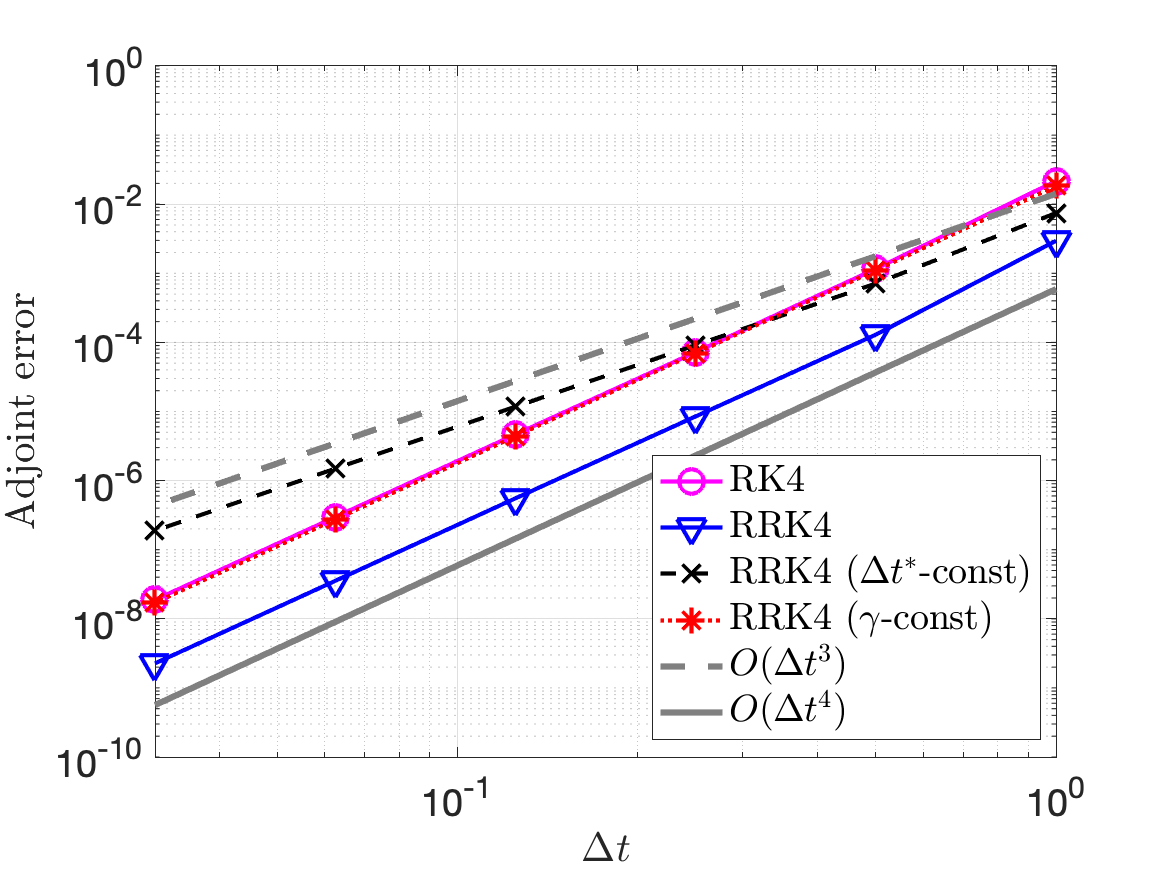}
		\caption{Using RK4.}
		\label{fig:nlpen_conv_RRK4_adj}
	\end{subfigure}
	\caption{Convergence of RK, and improperly linearized RRK, discretizations of the adjoint nonlinear pendulum problem; $T=2$.}
	\label{fig:nlpen_conv_RRK_adj}
\end{figure}

\subsection{Qualitative behavior of adjoint solutions}

The adjoint convergence plots in figure \ref{fig:nlpen_conv_RRK_adj} not only demonstrate that the adjoint RRK method (with proper linearization) is consistent but is also the most accurate out of all of the other methods, including standard RK.
We explore this in these next set of experiments, using again the nonlinear pendulum as our state equations with a much larger final time of $T=200$.

In figure~\ref{fig:norm_adj} and \ref{fig:nlpen_adj_norm_big}, we plot the norm of the adjoint solution as a function or time.
Note that the plots are presented with the time axis in reverse, in spirit with the back-propagation of adjoint numerical methods.
We see that all of the methods for computing the adjoint solution yield consistent results in the case where  we use both RK4 and a small step-size of $\Delta t=0.1$, figure~\ref{fig:nlpen_adj_norm_RRK4}.
However, when using RK2, we begin to see standard RK2 and RRK2 ($\gamma$-constant) deteriorate, figure~\ref{fig:nlpen_adj_norm_RRK2}.
Standard RK2 yields significant differences from the onset but eventually recovers a similar growth in the adjoint solution.
Conversely, RRK2 ($\gamma$-constant) yields reasonably accurate results as it begins stepping backwards in time, but soon after becomes erratic. 
RRK with proper linearization and with $\Delta t^*$-constant produce accurate results for all three choices of RK schemes, at least for the smaller step-size, $\Delta t=0.1$.

For the larger step-size, figure~\ref{fig:nlpen_adj_norm_big}, we can observe more significant differences between RRK with proper linearization and the $\Delta t^*$-constant case, which becomes more apparent as we progress backwards in time.
Moreover, the $\gamma$-constant case results in a substantial growth of the adjoint solution by the time we arrive at the initial time.
The standard adjoint RK4 method yields highly inaccurate results, in part due to the numerical dissipation observed in its forward solution, see figure~\ref{fig:nlpen_fwd_sol_big}.

\begin{figure}[!h]
	\centering
	\begin{subfigure}{0.45\textwidth}
		\centering
		\includegraphics[width=\textwidth]{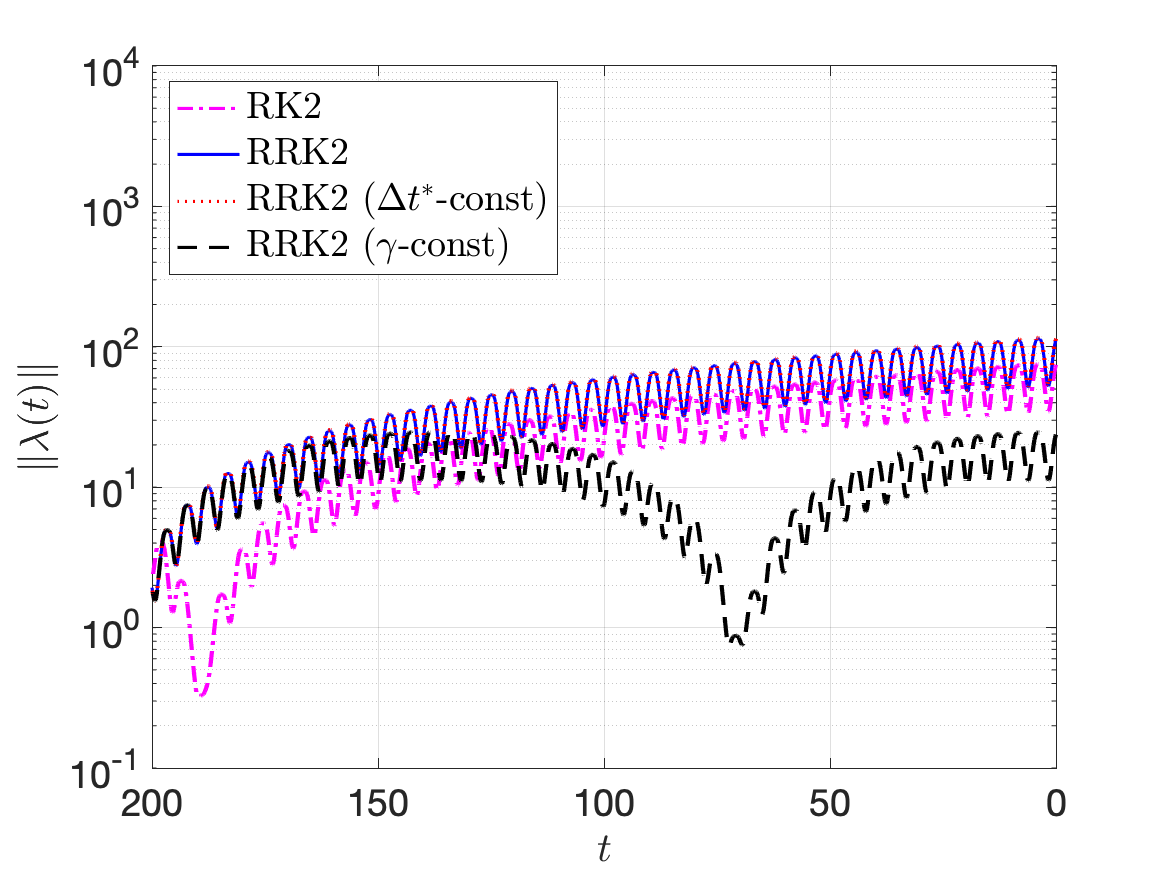}
		\caption{Using RK2.}
		\label{fig:nlpen_adj_norm_RRK2}
	\end{subfigure}
	\hspace{1em}
	\begin{subfigure}{0.45\textwidth}
		\centering
		\includegraphics[width=\textwidth]{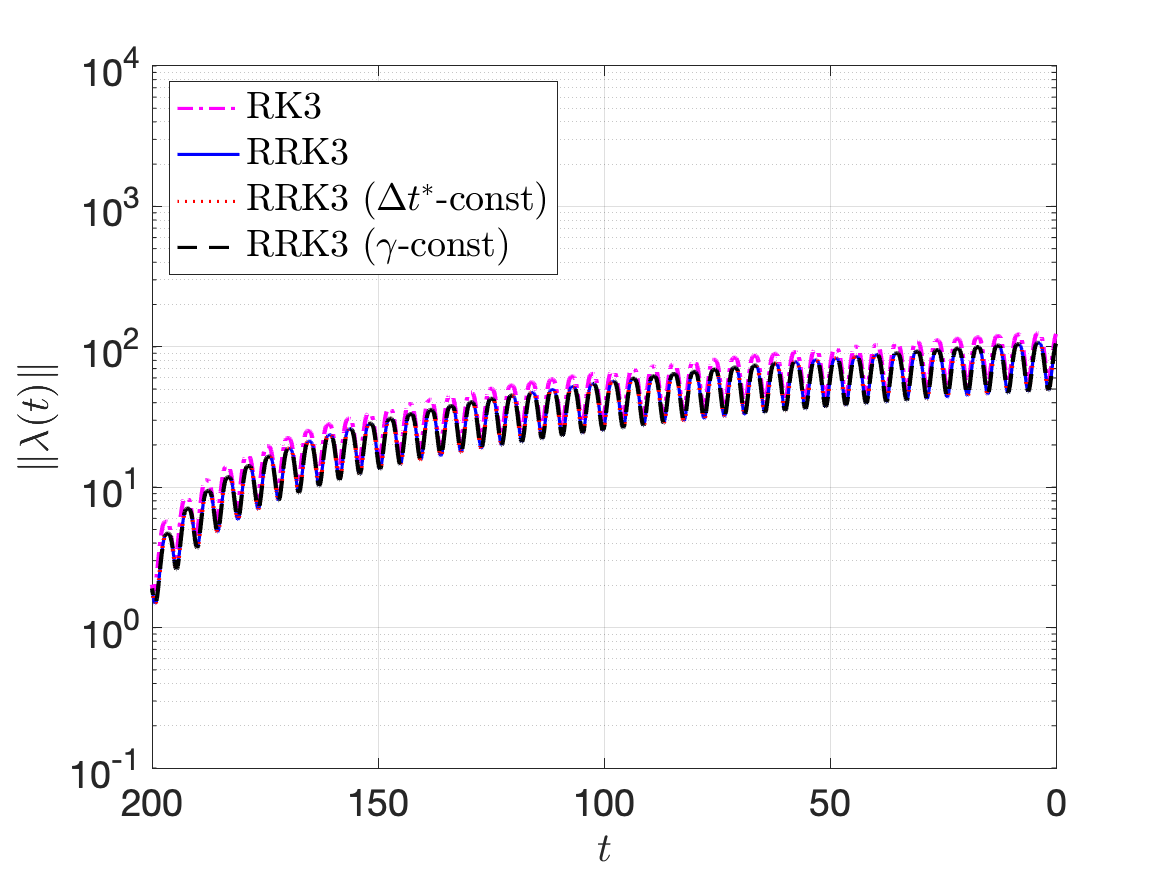}
		\caption{Using RK3.}
		\label{fig:nlpen_adj_norm_RRK3}
	\end{subfigure}
	\\
	\begin{subfigure}{0.45\textwidth}
		\centering
		\includegraphics[width=\textwidth]{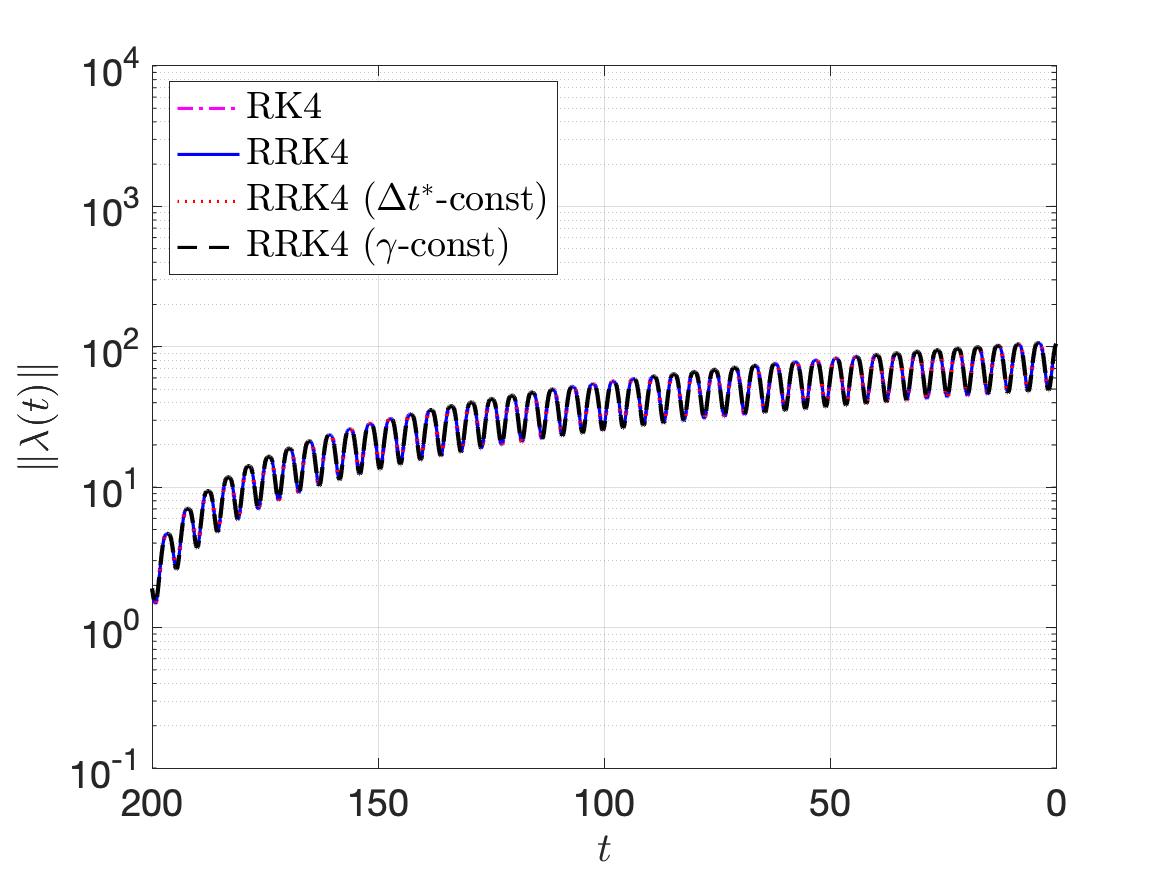}
		\caption{Using RK4.}
		\label{fig:nlpen_adj_norm_RRK4}
	\end{subfigure}

	\caption{Norm of adjoint solutions to the nonlinear pendulum model; $\Delta t=0.1$ and $T=200$.}
	\label{fig:norm_adj}
\end{figure}

\begin{figure}[!h]
	\centering
	\begin{subfigure}{0.45\textwidth}
		\centering
		\includegraphics[width=\textwidth]{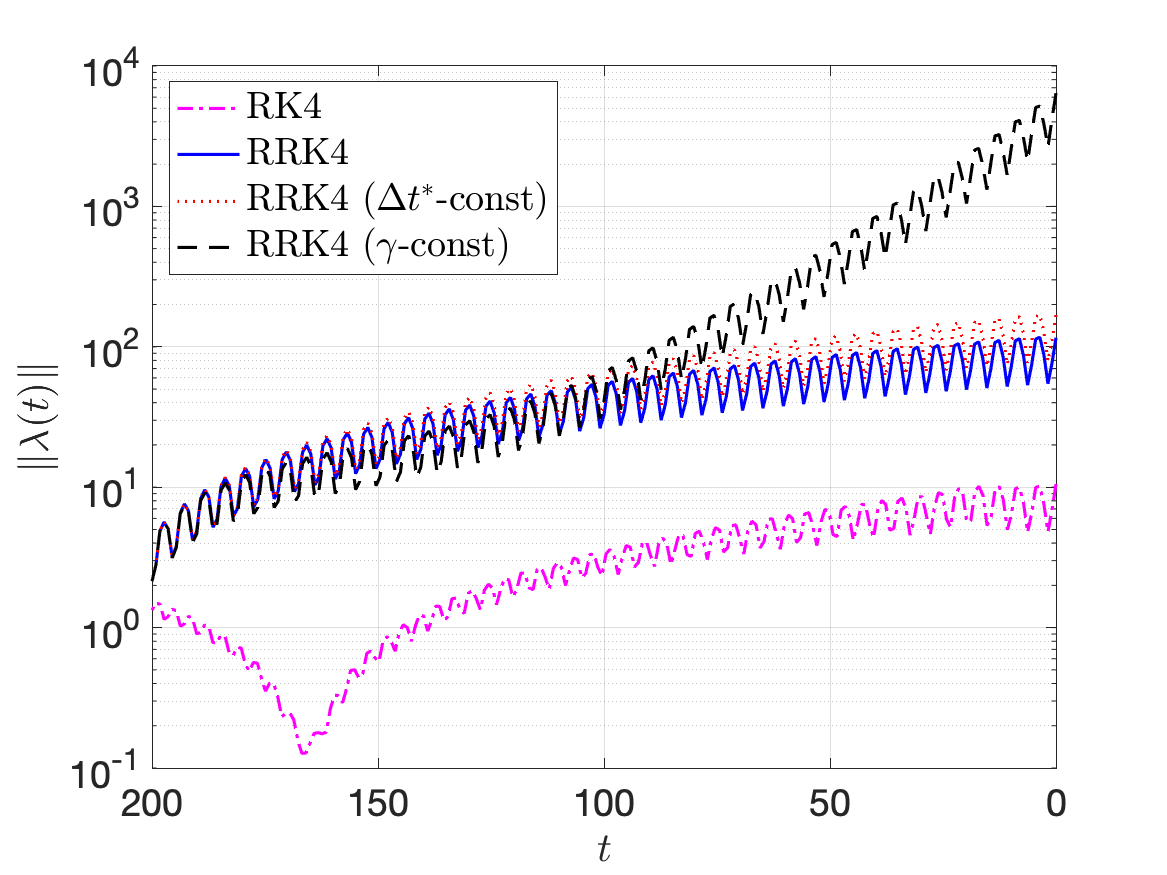}
		\caption{Norm of adjoint solutions.}
		\label{fig:nlpen_adj_norm_big}
	\end{subfigure}
	\hspace{1em}
	\begin{subfigure}{0.45\textwidth}
		\centering
		\includegraphics[width=\textwidth]{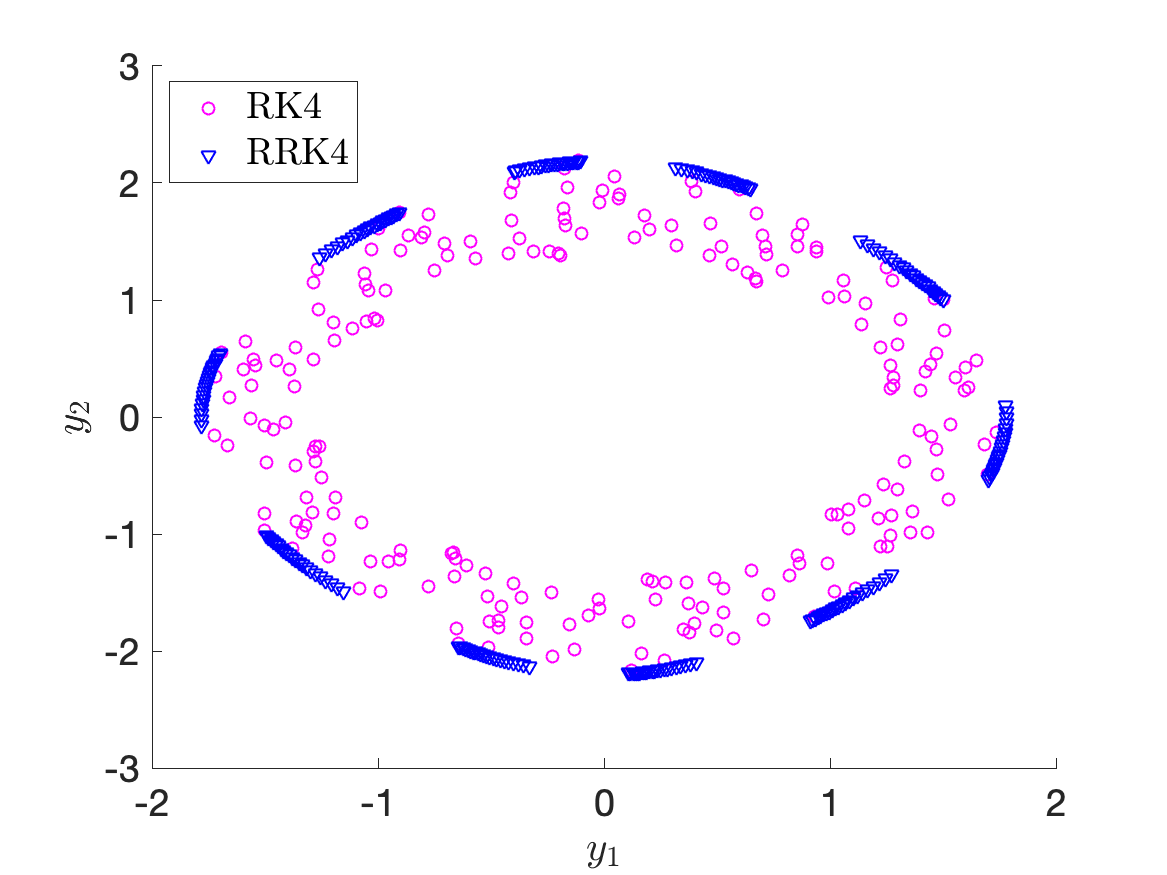}
		\caption{Phase plot of forward solutions.}
		\label{fig:nlpen_fwd_sol_big}
	\end{subfigure}

	\caption{Norm of adjoint solutions and scattered plot of forward solution to the nonlinear pendulum model, using RK4; $\Delta t=0.9$ and $T=200$.}
	\label{fig:nlpen_big}
\end{figure}

\subsection{Preservation of time-symmetry}

The next set of results verify the time-symmetry property of RRK schemes for model skew-symmetric problems \ref{eq:skewsym}, as detailed in theorem \ref{thm:RRKskew2}.
To showcase time-symmetry (or the lack thereof) we solve the skew-symmetric problem and use the numerical solution at the final time as the final-time condition for the adjoint solution.
If a scheme is time symmetric then the initial condition and the numerical adjoint solution should coincide up to machine precision.
Figure \ref{fig:time_reverse} plots the error
\[
	\frac{\| \fntb \lambda_0 - \fntb y_{\rm init}\|}{\|\fntb y_{\rm init}\|}
\]
versus $\Delta t$.

Numerical results demonstrate the ability for RRK to preserve time-symmetry in its adjoint, as observed by the small errors for both explicit and implicit RK schemes.
Time-symmetry is violated in both the standard RK and RRK with $\fntb \gamma$-constant cases.
Interestingly enough, numerical results show that RK4 and RK2 (e.g., even order RK schemes) converge at a rate higher than anticipated (fifth and third order convergence respectively) while RK3 and DIRK3 maintain their third order convergence.
The authors are unaware as to why RK4 and RK2 exhibit this super convergence behavior.
It is also observed that the $\Delta t^*$-constant case demonstrates time-symmetry, which can be explain by the fact that the terms appearing in the proper linearization (specifically scalar $\rho_*$ and $\xi_*$ in Lemma \ref{thm:linRRK} and \ref{thm:adjRRK}) that would otherwise be missing in the $\Delta t^*$-constant case actually zero-out for this linear problem.
In other words, proper linearization and the $\Delta t^*$-constant case are equivalent for this linear problem with square entropy. 

\begin{figure}[!h]
	\centering
	\begin{subfigure}{0.45\textwidth}
		\centering
		\includegraphics[width=\textwidth]{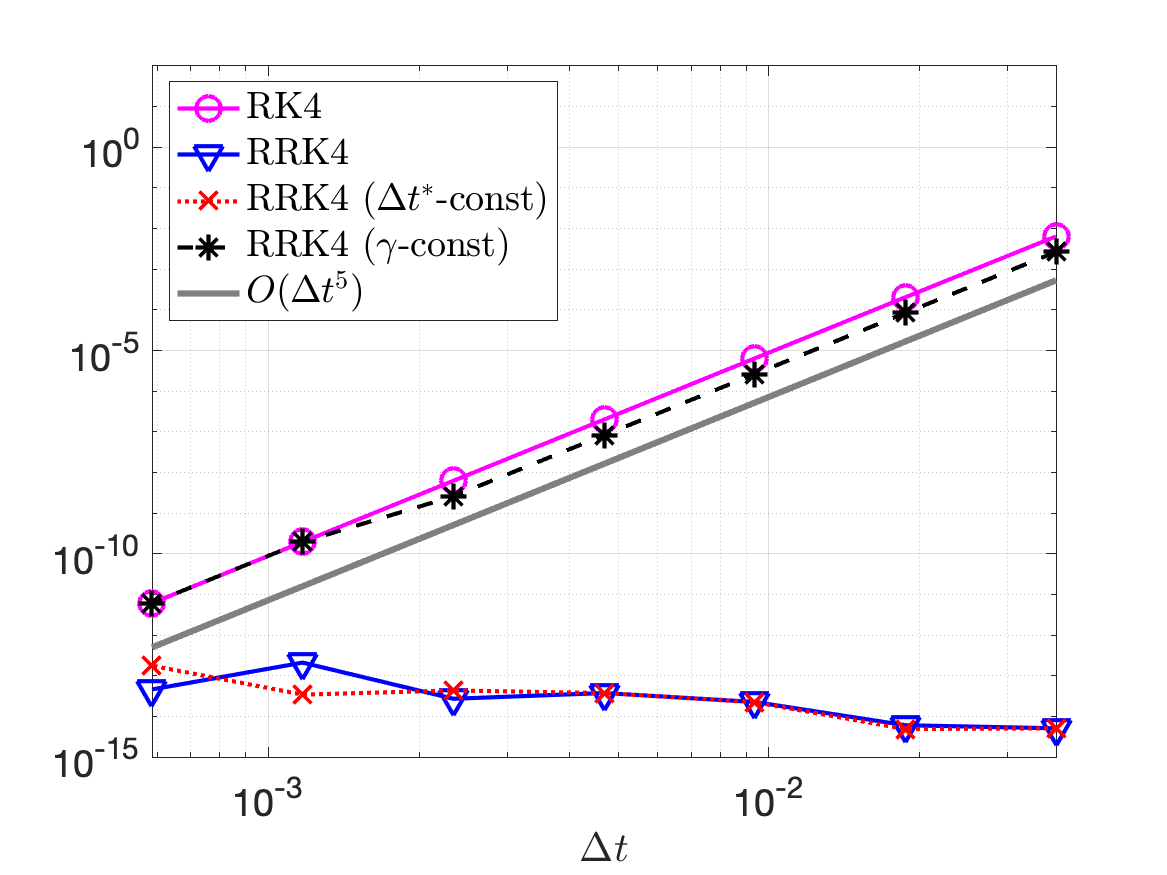}
		\caption{RK4 results.}
		\label{fig:ssymm_RK4}
	\end{subfigure}
	\hspace{1em}
	\begin{subfigure}{0.45\textwidth}
		\centering
		\includegraphics[width=\textwidth]{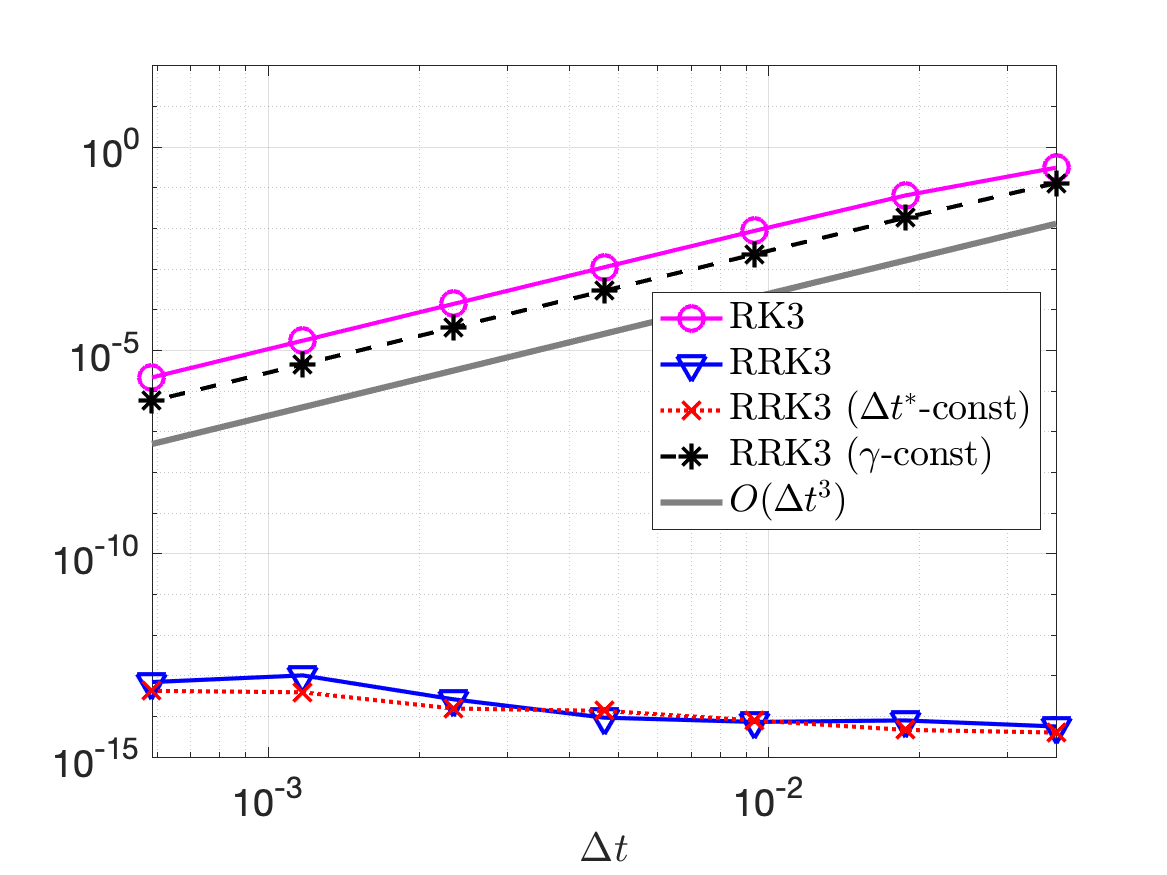}
		\caption{RK3 results.}
		\label{fig:ssymm_RK3}
	\end{subfigure}
	\\
	\begin{subfigure}{0.45\textwidth}
		\centering
		\includegraphics[width=\textwidth]{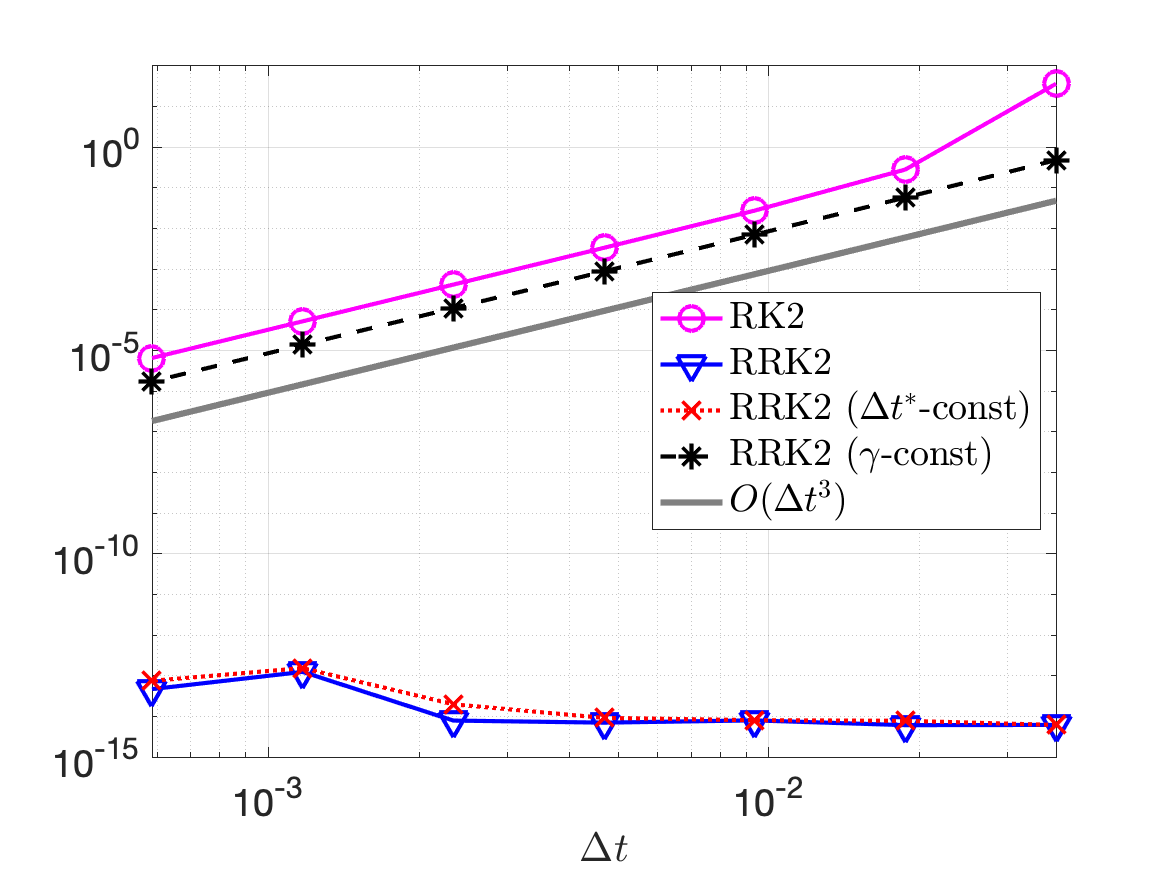}
		\caption{RK2 results.}
		\label{fig:ssymm_RK2}
	\end{subfigure}
	\hspace{1em}
	\begin{subfigure}{0.45\textwidth}
		\centering
		\includegraphics[width=\textwidth]{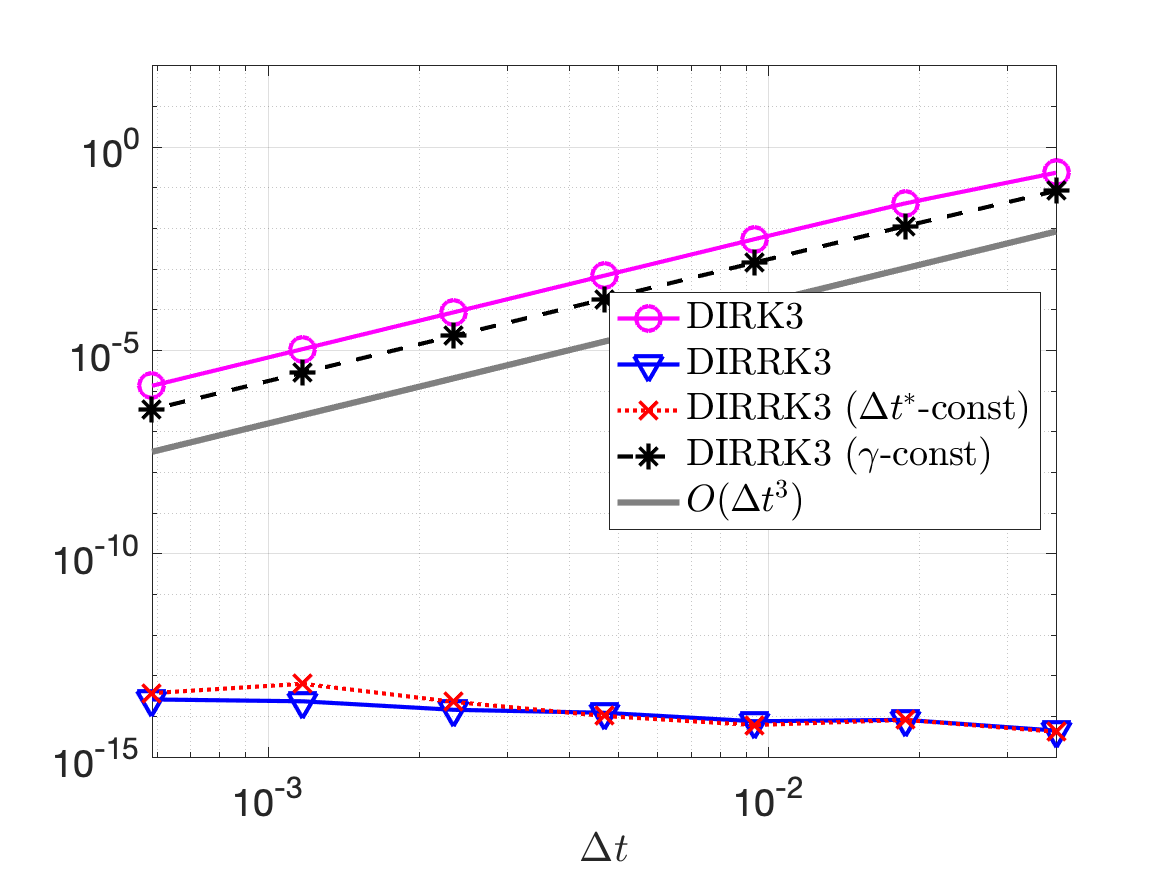}
		\caption{DIRK3 results.}
		\label{fig:ssymm_DIRK3}
	\end{subfigure}
	\caption{Time-symmetry results.}
	\label{fig:time_reverse}
\end{figure}

\pagebreak
\section{Conclusion}
\label{sec:Conclusion}
We have presented the discrete linearization and adjoint of relaxation RK methods.
Our approach is based on implicit differentiation and a global matrix representation of the time-stepping equations.
Even though the relaxation parameter is a nonlinear function of the state variables, we are able to prove that the relaxation RK method is equivalent to its linearization when applied to a skew-symmetric linear problem.
Moreover, the relaxation method is proven to be \textit{time-symmetric} for explicit and diagonally implicit RK schemes on skew-symmetric linear problems, in the sense that the adjoint time-step reverses the forward time-step.
Numerical results also demonstrate the importance of proper linearization, in particular, the importance of taking into account the relaxation parameter, and the corrected final step-size for our implementation of RRK.
Numerical results also show that the discrete RRK adjoint is not only consistent (with optimal convergence), but is more accurate in computing adjoint solutions.

\section*{Appendix}
\label{sec:Appendix}

In this appendix we present in detail derivations of linearization and adjoint computations for RK and its relaxation variant using a matrix representation.
We use an approach and notation similar to \cite{Rothauge_2016}.
The plan is to interpret time-stepping algorithms as solutions to global matrix-vector systems and use the Jacobians of said systems to deduce a time stepping scheme for the discrete linearization and adjoint.

We first recap and introduce some notation.
\begin{itemize}

\item 
$N$ is the dimension of the state vector $\mb y(t)$ in IVP \ref{eq:IVP};

\item 
$K$ is the total number of time steps taken by a given time-stepping scheme;

\item 
$s$ is the number of internal stages for an RK method;

\item
$\fntb A_s, \fntb b_s, \fntb c_s$ are the coefficients of a specified $s$-stage RK method;

\item
The concatenation of vectors indexed by internal stages is denoted by simply removing the internal stage index, e.g.,
\[
	\fntb Y_k \DEF \mat{\fntb Y_{k,1} \\ \vdots \\ \fntb Y_{k,s}} \in\mathbb R^{sN};
\]

\item
Vectors of size $\ol N \DEF N + N(s+1)K$ are denoted using bold font and an overline and have components denoted as follows:
\[
	\fnto y \DEF \mat{\fntb y_0 \\ \fntb Y_1 \\ \fntb y_1 \\ \vdots \\ \fntb Y_{K} \\ \fntb y_{K}},
\]
with $\fntb y_k\in\mathbb R^N$ and $\fntb Y_k\in\mathbb R^{sN}$;

\item 
Matrix $\fnts \chi_k\in\mathbb R^{\ol N\times N}$ is defined such that $\fnts \chi_k^\top \fnto y = \fntb y_k$, extracting the $k$-th step vector.
Moreover, for a given $\fntb v\in\mathbb R^N$, then $\fnto v \DEF \fnts \chi_k \fntb v$ is vector of length $\ol N$ with zero entries everywhere except at $\fntb v_k = \fntb v$;

\item
$\fntb I_{M}$ denotes the $M\times M$ identity matrix;

\item
$\fntb 0_{M}$ is the zero vector of dimension $M$;

\item 
$\fntb 0_{M_1\times M_2}$ is the $M_1\times M_2$ zero matrix;

\item
$\otimes$ denotes the Kronecker product.

\end{itemize}

\subsection*{RK Matrix-representation}

A single step of the RK method, as specified by equations \ref{eq:RK}, can be written in matrix form as
\[
	\mat{ -\fntb C & \fntb I_{sN} &  \\ -\fntb I_N &   & \fntb I_N }
	\mat{\fntb y_{k-1} \\ \fntb Y_{k} \\ \fntb y_k } - \mat{ \fntb A \fntb F_{k} \\ \fntb B^\top \fntb F_{k}} 
	= \mat{ \fntb 0_{sN} \\ \fntb 0_N}
\]
where
\begin{itemize}
	\item $\fntb A \DEF \Delta t \fntb A_s \otimes \fntb I_N \in\mathbb R^{sN\times sN}$,
	\item $\fntb B \DEF \Delta t \fntb b_s \otimes \fntb I_N \in \mathbb R^{sN\times N}$,
	\item $\fntb C \DEF \fntb 1_s \otimes \fntb I_N \in \mathbb R^{sN\times N}$, with $\fntb 1_s \DEF (1,...,1)^\top\in\mathbb R^{s},$
	\item $\fntb F_k$ is the concatenation of the $\fntb F_{k,i}\DEF \mb f(\fntb Y_{k,i},t_{k-1}+c_i \Delta t)$, and subsequently can be viewed as a vector valued function of $\fntb Y_k$.
\end{itemize}
The RK method as a whole can be represented as a concatenation of the matrix systems above, resulting in a global system of time-stepping equations:
\[
	\fntb E(\fnto y) \DEF \fntb L \fnto y - \fntb N(\fnto y) - \fnts \chi_0 \fntb y_{\rm init} = \fnto 0
\]
with
\[
	\fntb L \DEF
	\MAT{\fntb I_N}{-\fntb C}{\fntb I_{sN}}{-\fntb I_N}{}{-\fntb C}{\fntb I_{sN}}{-\fntb I_N}{},
	\quad
	\fntb N(\fnto y) \DEF
	\mat{\fntb 0_{N} \\ \hline \fntb A \fntb F_1 \\ \fntb B^\top \fntb F_1\\ \hline \fntb A\fntb F_2 \\ \fntb B^\top \fntb F_2 \\ \hline  \vdots }.
\]
We make some remarks and observations:
\begin{itemize}
	\item 
	Boxes are meant to help visually separate blocks associated with different time steps in both $\fntb L$ and $\fntb N(\fnto y)$.
	
	\item 
	The unboxed $\fntb I_N$ in the top left corner of $\fntb L$ is related to the enforcement of the initial condition.

	\item 
	$\fntb L\in\mathbb R^{\ol N\times \ol N}$ is lower unit triangular, though not quite block diagonal due to some slight overlap in columns.

	\item 
	Given the repeating block structure of matrices presented here, we only write out the blocks associated the initial/final conditions and two subsequent time steps.
	Dots indicate a repeating pattern with the understanding that the block structure repeats $K$ times with appropriate indexing when relevant.

	\item 
	$\fntb N:\mathbb R^{\ol N} \to \mathbb R^{\ol N}$ is block lower triangular in the sense that the $k$-th $N(s+1)$ block of the output does not depend on the $j$-th $N(s+1)$ block of the input, for $j>k$.
	For example, the $N(s+1)$ block of $\fntb N(\fnto y)$ associated with the $k$-th time step, is
\[
	\mat{\fntb A\fntb F_k \\ \fntb B^\top\fntb F_k}
\]
which only depends on $\fntb Y_k$, and not on $\fnto y$ at later time steps.
\end{itemize}

Linearized RK formulas \ref{eq:linRK} are derived by computing the Jacobian, $\fntb E'$, and interpreting the solution to linear system $\fntb E'(\fnto y) \fnto \delta = \fnto w$ (via forward substitution) as a time-stepping algorithm.
The Jacobian of the global time-stepping equation operator $\fntb E(\fnto y)$ is given by
\[
	\fntb E'(\fnto y) = \fntb L - \fntb N'(\fnto y)
\]
with
\[
	\fntb N'(\fnto y) = 
\MAT{\fntb 0_{N\times N}}{}{\fntb A\fntb J_1}{}{\fntb B^\top \fntb J_1}{}{\fntb A\fntb J_2}{}{\fntb B^\top \fntb J_2}
\]
\begin{align*}
	\fntb J_k &\DEF {\rm diag}(\fntb J_{k,1}, \cdots, \fntb J_{k,s}) \in \mathbb R^{sN\times sN},\\
	\fntb J_{k,i} &\DEF \frac{\partial \mb f}{\partial \mb y}(\fntb Y_{k,i},t_{k-1}+c_i\Delta t) \in \mathbb R^{N\times N}.
\end{align*}
We see that the block lower triangular structure of the operator $\fntb N$ yields a proper block lower triangular Jacobian $\fntb N'$.
Putting it together, we have
\[
	\fntb E'(\fnto y) = \MAT{\fntb I_N}{-\fntb C}{\fntb I_{sN}-\fntb A\fntb J_1}{-\fntb I_N}{-\fntb B^\top \fntb J_1}{-\fntb C}{\fntb I_{sN}-\fntb A\fntb J_2}{-\fntb I_N}{-\fntb B^\top \fntb J_2}
\]
In particular, each time step is associated with solving the following system for $\fnts \Delta_k$ and $\fnts \delta_k$, with $\fnts \delta_{k-1}$ given by the previous time step:
\[
\begin{split}
	& \quad 
		\mat{-\fntb C & \fntb I_{sN} - \fntb A \fntb J_k \\ -\fntb I_N & -\fntb B^\top \fntb J_k & \fntb I_N}
		\mat{\fnts \delta_{k-1} \\ \fnts \Delta_k \\ \fnts\delta_{k}} =
		\mat{\fntb W_k\\ \fntb w_k},\\
	\implies & \quad
	\left\{\begin{split}
		\fnts \Delta_{k,i} &= \fnts \delta_{k-1} + \Delta t \sum_{j=1}^s a_{ij} \fntb J_{k,j}\fnts \Delta_{k,j} + \fntb W_{k,i},\quad i=1,...,s, \\
		\fnts \delta_k &= \fnts \delta_{k-1} + \Delta t\sum_{i=1}^s b_i \fntb J_{k,i} \fnts \Delta_{k,i} + \fntb w_k.
	\end{split}\right.
\end{split}
\]

Adjoint RK formulas \ref{eq:adjRK} are derived in a similar fashion, but with the transpose of $\fntb E'$, which results in a block upper triangular matrix,
\[
	\fntb E'(\fnto y)^\top = \MATT{\fntb I_N}{-\fntb C^\top}{-\fntb I_N}{\fntb I_{sN}-\fntb J^\top_{K-1}\fntb A^\top}{-\fntb J^\top_{K-1}\fntb B}{-\fntb C^\top}{-\fntb I_N}{\fntb I_{sN}-\fntb J^\top_{K}\fntb A^\top}{-\fntb J^\top_{K}\fntb B}.
\]
Analogous to $\fntb L$, we see a repeating block structure with overlapping columns, though with an identity block at the lower right corner, associated with the final time condition.
We interpret the solution to linear system $\fntb E'(\fnto y)^\top \fnto\lambda=\fnto w$ (via back substitution) as a time-stepping algorithm running in reverse time.
Each time step is associated with solving the following system for $\fnts \Lambda_{k}$ and $\fnts \lambda_{k-1}$, with $\fnts \lambda_{k}$ given by the previous adjoint time step:
\[
		\mat{ \fntb I_N & -\fntb C^\top & -\fntb I_N \\ & \fntb I_{sN}-\fntb J^\top_{k}\fntb A^\top & -\fntb J^\top_{k}\fntb B}
		\mat{\fnts \lambda_{k-1} \\ \fnts \Lambda_k \\ \fnts \lambda_k} =
		\mat{\fntb w_{k-1} \\ \fntb W_k},
\]
\[
	\implies \quad
		\left\{\begin{split}
			\fnts \lambda_{k-1} &= \fnts \lambda_k + \sum_{i=1}^s \fnts \Lambda_{k,i} + \fntb w_{k-1}, \\
			\fnts \Lambda_{k,i} &= b_i \Delta t \fntb J_{k,i}^\top \fnts \lambda_{k} + \Delta t \sum_{j=1}^s a_{ji}\fntb J^\top_{k,i} \fnts \Lambda_{k,j} + \fntb W_{k,i}, \quad i=1,...,s.
		\end{split}\right.
\]

\subsection*{IDT matrix representation}

The matrix representation of the IDT method is very similar to what we derived for RK,
\[
	\fntb E(\fnto y,\fnts \gamma) \DEF \fntb L\fnto y - \fntb N(\fnto y,\fnts \gamma) - \fnts \chi_0 \fntb y_{\rm init} = \fnto 0
\]
where the relaxation parameters $\fnts \gamma = (\gamma_1,...,\gamma_k)$ appear on the (nonlinear) term $\fntb N$ only, i.e.,
\begin{equation}\label{eq:N_IDT}
	\fntb N(\fnto y,\fntb\gamma) \DEF \mat{ \fntb 0_{N} \\ \hline \fntb A\fntb F_1 \\ \gamma_1\fntb B^\top \fntb F_1\\ \hline \fntb A\fntb F_2 \\ \gamma_2\fntb B^\top\fntb F_2 \\ \hline \vdots }.
\end{equation}
Recall that $\gamma_k$ is defined as the positive root near $1$ (for $\Delta t$ small enough) of the root function $r(\gamma; \fntb y_{k-1},\fntb Y_k)$, equation \ref{eq:root}.
In other words the relaxation parameters depend implicitly on $\fnto y$, i.e., $\fnts \gamma=\fnts \gamma(\fnto y)$.
Let $\wt{\fntb E}(\fnto y)$ denote the reduced state-equation operator, that is,
\[
	\wt{\fntb E}(\fnto y) \DEF \fntb E(\fnto y,\fnts\gamma(\fnto y)).
\]
The Jacobian is given by
\[
	\wt{\fntb E}'(\fnto y) 
	= \fntb L - \underbrace{\left(\frac{\partial\fntb N}{\partial\fnto y}(\fnto y,\fnts \gamma(\fnto y)) \red{\, + \, \frac{\partial\fntb N}{\partial\fnts \gamma}(\fnto y,\fnts \gamma(\fnto y)) \fnts \gamma'(\fnto y)}\right)}_{\wt{\fntb N}'(\fnto y)}.
\]
In particular,
\[
\wt{\fntb N}'(\fnto y) = 
\MAT{\fntb 0_{N\times N}}{}{\fntb A\fntb J_1}{\red{\fnts \Gamma_{y,1}}}{\gamma_1\fntb B^\top\fntb J_1 \red{\,+\, \fnts \Gamma_{Y,1}}}{}{\fntb A\fntb J_2}{\red{\fnts \Gamma_{y,2}}}{\gamma_2\fntb B^\top\fntb J_2 \red{\,+\, \fnts \Gamma_{Y,2}}}
\]
where
\[
	\fnts \Gamma_{y,k} \DEF \fntb B^\top\fntb F_k (\nabla_{y}\gamma_k)^\top,
	\quad 
	\fnts \Gamma_{Y,k} \DEF \fntb B^\top\fntb F_k (\nabla_{Y}\gamma_k)^\top
\]
are associated with the term $\red{\frac{\partial \fntb N}{\partial \fnto y}\fnts \gamma'}$, which we highlight using red text.
Again, gradient terms in $\fnts \Gamma_{y,k}$ and $\fnts \Gamma_{Y,k}$ correspond to gradients of $\gamma_k$ with respect to $\fntb y_{k-1}$ and $\fntb Y_k$ respectively, and are computed via implicit differentiation; see equations \ref{eq:gradgamma} and \ref{eq:dr}.

The Jacobian matrix for IDT is thus given by
\[
\wt{\fntb E}'(\fnto y) = 
	\MAT{\fntb I_N}{-\fntb C}{\fntb I_{sN}-\fntb A\fntb J_1}{-\fntb I_N\red{\,-\,\fnts \Gamma_{y,1}}}{-\gamma_1\fntb B^\top\fntb J_1 \red{\,-\,\fnts \Gamma_{Y,1}}}{-\fntb C}{\fntb I_{sN}-\fntb A\fntb J_2}{-\fntb I_N \red{\,-\,\fnts \Gamma_{y,2}}}{-\gamma_2\fntb B^\top\fntb J_2 \red{\,-\,\fnts \Gamma_{Y,2}}}.
\]
Solving $\wt{\fntb E}'(\fnto y)\fnto \delta = \fnto w$ via forward substitution results in solving at each time step the following system for $\fnts \Delta_k$ and $\fnts \delta_k$, with $\fnts \delta_{k-1}$ given by the previous time step, deriving the linearized IDT formulas in lemma \ref{thm:linIDT}:
\[
\begin{split}
	&\quad
	\mat{-\fntb C & \fntb I_{sN}-\fntb A\fntb J_k \\ -\fntb I_N \red{ - \fnts \Gamma_{y,k}} & -\gamma_k\fntb B^\top \fntb J_k \red{ - \fnts \Gamma_{Y,k}} & \fntb I_N}
	\mat{\fnts \delta_{k-1}\\ \fnts \Delta_k \\ \fnts \delta_k} =
	\mat{\fntb W_k\\ \fntb w_k} \\
	\implies &\quad
	\left\{\begin{split}
		\fnts \Delta_{k,i} &= \fnts \delta_{k-1} + \Delta t \sum_{j=1}^s a_{ij} \fntb J_{k,j} \fnts \Delta_{k,j} + \fntb W_{k,i}, \quad i=1,...,s,\\
		\fnts \delta_k &= \fnts \delta_{k-1} + \gamma_k \Delta t\sum_{i=1}^s b_i \fntb J_{k,i} \fnts \Delta_{k,i} \red{\,+\, \fnts \Gamma_{y,k}\fnts \delta_{k-1} + \fnts \Gamma_{Y,k}\fnts \Delta_k } + \fntb w_k.
	\end{split}\right.
\end{split}
\]
Note that
\[
	\fnts \Gamma_{y,k}\fnts \delta_{k-1} + \fnts \Gamma_{Y,k}\fnts \Delta_k =
	\rho_k \Delta t\sum_{i=1}^s b_i \fntb F_{k,i}
\]
with scalar $\rho_k \DEF \nabla_{y} \gamma_k^\top\fnts \delta_{k-1}+ \nabla_{Y} \gamma_k^\top\fnts \Delta_k$.

The transpose of the Jacobian is given by
{\footnotesize
\[
\wt{\fntb E}'(\fnto y)^\top = 
	\MATT{\fntb I_N}{-\fntb C^\top}{-\fntb I_N \red{-\fnts \Gamma^\top_{y,K-1}}}{\fntb I_{sN}-\fntb J_{K-1}^\top\fntb A^\top}{-\gamma_{K-1} \fntb J^\top_{K-1}\fntb B \red{- \fnts \Gamma^\top_{Y,K-1}}}{-\fntb C^\top}{-\fntb I_N \red{-\fnts \Gamma^\top_{y,K}}}{\fntb I_{sN}-\fntb J^\top_{K}\fntb A^\top}{-\gamma_{K}\fntb J^\top_{K}\fntb B \red{- \fnts \Gamma^\top_{Y,K}}}.
\]}
Solving $\wt{\fntb E}'(\fnto y)^\top\fnto \lambda = \fnto w$ via back substitution results in solving at each time step the following system for $\fnts \Lambda_k$ and $\fnts \lambda_k-1$, with $\fnts \lambda_{k}$ given by the previous time step, deriving the adjoint IDT formulas in lemma \ref{thm:adjIDT}:
\[
\begin{split}
	& \quad
	\mat{\fntb I_N & -\fntb C^\top & - \fntb I_N \red{\,-\, \fnts \Gamma_{y,k}^\top} \\ & \fntb I_{sN}-\fntb J_{k}^\top \fntb A^\top & -\gamma_k \fntb J_k^\top \fntb B \red{\,-\, \fnts \Gamma_{Y,k}^\top}}
	\mat{\fnts \lambda_{k-1} \\ \fnts \Lambda_k \\ \fnts \lambda_k} = 
	\mat{\fntb w_{k-1} \\ \fntb W_k} \\ 
	\implies &\quad
	\left\{\begin{split}
		\fnts \lambda_{k-1} &= \sum_{i=1}^s \fnts \Lambda_{k,i} + \fnts \lambda_k  \red{\,+\, \fnts \Gamma_{y,k}^\top \fnts \lambda_k} + \fntb w_{k-1}, \\
		\fnts \Lambda_{k,i} &=  \Delta t \, \fntb J_{k,i}^\top \sum_{j=1}^s a_{ji} \fnts \Lambda_{k,j} + \gamma_k b_i \Delta t\, \fntb J_{k,i}^\top \fnts \lambda_k  \red{ \,+\, \fnts \Gamma_{Y,k}^\top \fnts \lambda_k} + \fntb W_{k,i}, \quad i=1,...,s.
	\end{split}\right.
\end{split}
\]
Note that
\begin{align*}
	\fnts \Gamma_{y,k}^\top \fntb \lambda_{k} 
		&= (\nabla_{y}\gamma_k) \; \fntb F_k^\top \fntb B \fnts \lambda_{k}
		 	= \xi_k (\nabla_{y}\gamma_k)\\
	\fnts \Gamma_{Y,k}^\top \fntb \lambda_{k} 
		&= (\nabla_{Y}\gamma_k) \fntb F_k^\top \fntb B \fnts \lambda_{k}
			= \xi_{k} (\nabla_{Y}\gamma_k)
\end{align*}
with scalar
\[
	\xi_{k} \DEF \fntb F_{k}^\top \fntb B \fnts \lambda_k = \Delta t \sum_{i=1}^s b_i \fntb F_{k,i}^\top \fnts \lambda_{k}.
\]

\subsection{RRK matrix representation}

The matrix representation for RRK is quite similar to what we derived for IDT:
\[
	\fntb E(\fnto y,\fnts \gamma, \Delta t^*) \DEF \fntb L \fnto y - \fntb N(\fnto y,\fnts \gamma, \Delta t^*) - \fnts \chi_0 \fntb y_{\rm init} = \fnto 0,
\]
where we have made the operator $\fntb N$ dependent on the modified step size $\Delta t^* \DEF T-t_{K-1}$ as well.
Only the last $N(s+1)$ rows of $\fntb N$, associated with the last time step, differ from what was presented in the IDT case (equation \ref{eq:N_IDT}).
These last  last $N(s+1)$ rows of $\fntb N$ are specified by
\[
	\mat{ \fntb A_*\fntb F_K \\ 
		 \gamma_K\fntb B_*^\top \fntb F_K }
\]
with
\[
\begin{split}
	\fntb A_* &\DEF \Delta t^* \fntb A_s \otimes \fntb I_N = \frac{\Delta t^*}{\Delta t} \fntb A,\\
	\fntb B_* &\DEF \Delta t^* \fntb b_s \otimes \fntb I_N = \frac{\Delta t^*}{\Delta t} \fntb B.
\end{split}
\]

Recall that in RRK we have $t_k = t_{k-1} + \gamma_k \Delta t$ for $k=1,...,K-1$, and hence
\[
	\Delta t^* = T - t_0 - \Delta t \sum_{\ell=1}^{K-1} \gamma_\ell
\]
which makes $\Delta t^*$ a function of $\fntb y_{\ell-1}$ and $\fntb Y_\ell$ for $\ell=1,...,K-1$, i.e., $\Delta t^* = \Delta t^*(\fnto y)$.
With this is mind, let $\wt{\fntb E}(\fnto y)$ denote the reduced state-equation operator, 
\[
	\wt{\fntb E}(\fnto y) \DEF \fntb E(\fnto y, \fnts\gamma(\fnto y), \Delta t^*(\fnto y)).
\]
Given how $\fntb N$ is modified in RRK, it follows that the Jacobian $\frac{d\wt{\fntb E}}{d\fnto y}$ will coincide with what we derived for IDT except at the last $N(s+1)$ rows.
In particular, computing
\begin{align*}
	\frac{d}{d\fnto y} \Big( \fntb A_* \fntb F_K \Big) 
		&=  \fntb A_* \frac{d \fntb F_K}{d\fntb y} 
		+ \fntb A \fntb F_K \paren{\frac{d}{d\fnto y}\frac{\Delta t^*}{\Delta t}},\\
	\frac{d}{d\fnto y} \paren{ \gamma_K \fntb B_*^\top \fntb F_K } 
		&= \gamma_K \fntb B_*^\top \frac{d\fntb F_K}{d\fnto y} 
		+ \fntb B_*^\top \fntb F_K \paren{\frac{d\gamma_K}{d\fnto y}} 
		+ \gamma_K \fntb B^\top \fntb F_K \paren{ \frac{d}{d\fnto y}\frac{\Delta t^*}{\Delta t}},
\end{align*}
will require the derivatives of $\gamma_K$ and $\Delta t^*$ with respect to $\fnto y$.

The derivatives of $\Delta t^*$ can be expressed in terms of derivatives of the relaxation parameters as follows:
\begin{align*}
	\frac{d}{d\fnto y} \frac{\Delta t^*}{\Delta t}(\fnto y)
	&= - \paren{
		\frac{\partial \gamma_1}{\partial \fntb y_1}, \;\, 
		\frac{\partial \gamma_1}{\partial \fntb Y_1}, \;\,
		\cdots,\;\,
		\frac{\partial \gamma_{K-1}}{\partial \fntb y_{K-1}}, \;\,
		\frac{\partial \gamma_{K-1}}{\partial \fntb Y_{K-1}}, \;\,
		\fntb 0_{N(s+2)}^\top
		}\eval{\fnto y}{}\\
	&= - \paren{ \nabla \gamma_1^\top,\;\, \cdots,\;\, \nabla\gamma_{K-1}^\top,\;\, \fntb 0_{N(s+2)}^\top },
\end{align*}
where
\[
	\nabla \gamma_k \DEF \mat{ \nabla_{y}\gamma_k \\ \nabla_{Y}\gamma_k}.
\]

An added complication is that $\Delta t^*$ is also the step size used in $r_K$, thus implying that $\gamma_K$ is dependent on information from all previous time steps.
As before, we can compute $\frac{d\gamma_K}{d\fnto y}$ via implicit differentiation, though we will have to compute partial derivatives of $\gamma_K$ with respect to $\fntb y_{\ell-1}$ and $\fntb Y_\ell$ for all $\ell=1,...,K$;
\begin{align*}
	\frac{\partial \gamma_K}{\partial \fntb y_{\ell-1}} &= -\paren{\frac{\partial r_K}{\partial \gamma}}^{-1} \frac{\partial r_K}{\partial \fntb y_{\ell-1}},\\
	\frac{\partial \gamma_K}{\partial \fntb Y_\ell} &= -\paren{\frac{\partial r_K}{\partial \gamma}}^{-1} \frac{\partial r_K}{\partial \fntb Y_{\ell}}.
\end{align*}
The partial derivatives of $r_K$ with respect to $\gamma$, $\fntb y_{K-1}$ and $\fntb Y_{K}$, evaluated at $(\gamma_K, \fntb y_{K-1}, \fntb Y_K)$ are as given in equation \ref{eq:dr}, with $k\mapsto K$ and $\Delta t\mapsto \Delta t^*$.
Just as in equation \ref{eq:gradgamma}, we use $\nabla_{y}\gamma_K$ and $\nabla_{Y}\gamma_K$ to denote the gradient of $\gamma_K$ with respect to $\fntb y_{K-1}$ and $\fntb Y_{K}$ respectively.
For the remaining $\ell=1,...,K-1$, 
\[
\begin{split}
	\frac{\partial r_K}{\partial \fntb y_{\ell-1}}(\fnts \gamma(\fnto y),\fnto y)
	&= \gamma_K \sum_{i=1}^s b_i \Big( \nabla\eta(\fntb y_{K}) - \nabla\eta(\fntb Y_{K,i})\Big)^\top \fntb F_{K,i} \frac{\partial\Delta t^*}{\partial \fntb y_{\ell-1}}(\fnto y)\\
	&= - \gamma_K \paren{\frac{\Delta t}{\Delta t^*} \frac{\partial r_K}{\partial \gamma}(\fnts\gamma(\fnto y),\fnto y)} (\nabla_{y} \gamma_\ell)^\top
\end{split}
\]
where we have used
\[
	\frac{\partial \Delta t^*}{\partial \fntb y_{\ell-1}}(\fnto y) = -\Delta t \, (\nabla_{y} \gamma_\ell)^\top.
\]
Similarly,
\[
\begin{split}
	\frac{\partial r_K}{\partial \fntb Y_{\ell,j}}(\fnts\gamma(\fnto y),\fnto y)
	&= \gamma_K \sum_{i=1}^s b_i \Big( \nabla\eta(\fntb y_{K}) - \nabla\eta(\fntb Y_{K,i})\Big)^\top \fntb F_{K,i} \frac{\partial\Delta t^*}{\partial \fntb Y_{\ell,j}}(\fnto y)\\
	&= - \gamma_K \paren{\frac{\Delta t}{\Delta t^*} \frac{\partial r_K}{\partial \gamma}(\fnts\gamma(\fnto y),\fnto y)} (\nabla_{Y} \gamma_{\ell,j})^\top.
\end{split}
\]
Thus,
\[
\begin{split}
	\frac{\partial \gamma_K}{\partial \fntb y_{\ell-1}}(\fnto y) &= \gamma_K \frac{\Delta t}{\Delta t^*} (\nabla_y \gamma_{\ell})^\top,\\
	\frac{\partial \gamma_K}{\partial \fntb Y_\ell}(\fnto y) &= \gamma_K \frac{\Delta t}{\Delta t^*} (\nabla_Y \gamma_{\ell})^\top.
\end{split}
\]
Putting it all together, we have 
\[
\begin{split}
	\frac{d\gamma_K}{d\fnto y}(\fnto y) 
	&= \Big( \gamma_K \tfrac{\Delta t}{\Delta t^*}\nabla \gamma_1^\top , \cdots, \gamma_K\tfrac{\Delta t}{\Delta t^*}\nabla\gamma_{K-1}^\top, \nabla \gamma_K^\top, \fntb 0_{N}^\top \Big) \\
	&= -\gamma_K \frac{\Delta t}{\Delta t^*} \paren{\frac{d}{d\fnto y} \frac{\Delta t^*}{\Delta t}(\fnto y)} 
	+ \paren{ \fntb 0_{N(s+1)(K-1)}^\top,\;\, \nabla \gamma_K^\top,\,\; \fntb 0_{N}^\top}.
\end{split}
\]
In summary,
\begin{align*}
	\frac{d}{d\fntb y} \Big(\fntb A_*\fntb F_K \Big)\eval{\fnto y}{}
	&= - \paren{ 
		\fntb A\fntb F_K \nabla\gamma_1^\top,\;\, 
		\cdots, \;\,
		\fntb A\fntb F_K \nabla\gamma_{K-1}^\top,\;\,
		\fntb 0_{sN\times N}, \;\,
		-\fntb A_*\fntb J_K,\;\,
		\fntb 0_{sN\times N}
		}\\
	\frac{d}{d\fnto y} \Big( \gamma_K \fntb B_*^\top \fntb F_K \Big)\eval{\fnto y}{}
	&= \paren{ \fntb 0_{N(s+1)(K-1)}^\top, \;\, \fnts \Gamma_{y,K}, \;\, \fnts\Gamma_{Y,K}, \;\, \fntb 0_{N}^\top}
\end{align*}
with
\[
	\fnts \Gamma_{y,K} =  \fntb B_*^\top \fntb F_{K} (\nabla_{y}\gamma_K)^\top,
	\quad
	\fnts \Gamma_{Y,K} =  \fntb B_*^\top \fntb F_{K} (\nabla_{Y}\gamma_K)^\top.
\]

We jump forward to interpreting the solution of $\wt{\fntb E}'(\fnto y)\fnto \delta = \fnto w$ via forward substitution as a time stepping scheme.
Again, the first $K-1$ steps are as given by IDT.
The last step, as shown in lemma \ref{thm:linIDT}, is derived from the solution of the following system for $\fnts \Delta_{K}$ and $\fnts \delta_K$, with $(\fnts\delta_{\ell-1},\fnts \Delta_{\ell})$ for $\ell=1,...,K-1$ given by the the previous time steps:
{\footnotesize
\[
	\mat{
	\blue{\fntb A \fntb F_K \nabla\gamma_1^\top} & 
	\cdots &
	\blue{\fntb A \fntb F_K \nabla\gamma_{K-1}^\top} &
	-\fntb C &
	\fntb I_{sN} - \fntb A_* \fntb J_K \\ 
	&&& -\fntb I_N \red{\,-\,\fnts \Gamma_{y,K}} & -\gamma_K \fntb B_*^\top \fntb J_K \red{\,-\, \fnts\Gamma_{Y,K}} & \fntb I_N}
	\mat{\mat{\fnts \delta_0 \\ \fnts \Delta_1}\\ \vdots \\ \mat{\fnts \delta_{K-2}\\ \fnts \Delta_{K-1}}\\ \fnts \delta_{K-1}\\ \fnts\Delta_K \\ \fnts \delta_K}
	= \mat{\fntb W_{K}\\ \fntb w_K}
\]}
{
\[
	\implies
	\left\{\begin{split}
		\fnts \Delta_{K,i} 
			&= 
			\fnts \delta_{K-1} 
			+ \Delta t^* \sum_{j=1}^s a_{ij} \fntb J_{K,j} \fnts \Delta_{K,j} 
			\blue{- \rho_* \Delta t\sum_{j=1}^s a_{ij}\fntb F_{K,j} } 
			+ \fntb W_{K,i}, \quad i=1,...,s,\\
		\fnts \delta_{K}
			&= \delta_{K-1} + \gamma_K \Delta t^* \sum_{i=1}^s b_i \fntb J_{K,i} \fnts \Delta_{K,i} 
			\red{\;+\; \fnts \Gamma_{y,K} \fnts \delta_{K-1} + \fnts \Gamma_{Y,K}\fnts \Delta_{K}}
			+ \fntb w_K
	\end{split}\right.
\]}
with scalar
\[
	\rho_* \DEF \sum_{\ell=1}^{K-1}\Big( (\nabla_{y}\gamma_\ell)^\top \fnts\delta_{\ell-1} + (\nabla_{Y}\gamma_\ell)^\top \fnts \Delta_{\ell} \Big).
\]
Similar to before, 
\[
	\fnts \Gamma_{y,K}\fnts \delta_{K-1} + \fnts \Gamma_{Y,K}\fnts \Delta_K 
		= \rho_K \Delta t^*\sum_{i=1}^s b_i \fntb F_{K,i},
\]
with scalar $\rho_K \DEF (\nabla_{y} \gamma_K)^\top\fnts \delta_{K-1}+ (\nabla_{Y} \gamma_K)^\top\fnts \Delta_K$.

The transpose of the Jacobian is
{\footnotesize
\[
\paren{\frac{d\wt{\fntb E}}{d\fnto y}(\fnto y)}^\top = 
\left( \begin{array}{c|cc|cc|c|}
	\multicolumn{1}{c}{\ddots} & 
	&
	\multicolumn{1}{c}{} & 
	& 
	\multicolumn{1}{c}{\vdots}
	\\
	\cline{2-5} &
	\fntb I_N & 
	\multicolumn{1}{c}{-\fntb C^\top} & 
	-\fntb I_N \red{\;-\; \fnts \Gamma_{y,K-1}^\top} & 
	\blue{\nabla_{y}\gamma_{K-1} \fntb F_{K}^\top \fntb A^\top}
	\\
	&
	&
	\multicolumn{1}{c}{\fntb I_{sN}-\fntb J_{K-1}^\top \fntb A_*^\top} & 
	-\gamma_{K-1} \fntb J_{K-1}^\top \fntb B_* \red{\;-\; \fnts \Gamma_{Y,K-1}^\top} & 
	\blue{\nabla_{Y}\gamma_{K-1} \fntb F_{K}^\top \fntb A^\top} 
	\\ 
	\cline{2-6} \multicolumn{1}{c}{}&
	&
	&
	\fntb I_N & \multicolumn{1}{c}{-\fntb C^\top} & 
	-\fntb I_N \red{\;-\; \fnts \Gamma_{y,K}^\top}
	\\
	\multicolumn{1}{c}{}&
	&
	&
	& 
	\multicolumn{1}{c}{\fntb I_{sN}-\fntb J_K^\top \fntb A_*^\top} & 
	-\gamma_K \fntb J_K^\top \fntb B_* \red{\;-\; \fnts \Gamma_{Y,K}^\top} \\ 
	\cline{4-6} \multicolumn{1}{c}{}&
	&
	\multicolumn{1}{c}{}&
	&
	\multicolumn{1}{c}{}& 
	\multicolumn{1}{c}{\fntb I_N}
\end{array}\right)
\]}
We now interpret the solution of $\wt{\fntb E}'(\fnto y)^\top\fnto \lambda = \fnto w$ via back substitution as a time stepping scheme.
Again, the last step (or first step in reverse-time) is as given by the adjoint IDT formulas in lemma \ref{thm:adjIDT}, but with $\Delta t \mapsto \Delta t^*$.
The remaining $K-1$ steps, as given in lemma \ref{thm:adjRRK}, are derived from the solution to the following systems for $\fnts \Lambda_k$ and $\fnts \lambda_{k-1}$, with $\fnts \lambda_k$ given by previous time step and $\fnts \Lambda_K$ given by the last step,
\[
\mat{
	\fntb I_N & 
	-\fntb C^\top & 
	-\fntb I_N \red{\,-\, \fnts \Gamma_{y,k}^\top} &
	\fntb 0_{N\times P} &
	\blue{\nabla_{y}\gamma_{k}^\top \fntb F_K^\top \fntb A^\top }
	\\ 
	& 
	\fntb I_{sN}-\fntb J_{k}^\top \fntb A^\top & 
	-\gamma_k \fntb J_k^\top \fntb B \red{\,-\, \fnts \Gamma_{Y,k}^\top} &
	\fntb 0_{sN\times P} &
	\blue{\nabla_{Y}\gamma_{k}^\top \fntb F_K^\top \fntb A^\top }
}
\mat{\fnts \lambda_{k-1} \\ \fnts \Lambda_k \\ \fnts \lambda_k\\ \vdots \\ \fnts \Lambda_K} = 
\mat{\fntb w_{k-1} \\ \fntb W_k}
\]
where $P=N(s+1)((K-1)-k)$, which gives
\[
\implies \quad
\left\{\begin{split}
	\fnts \lambda_{k-1} &= 
		\sum_{i=1}^s \fnts \Lambda_{k,i} 
		+ \fnts \lambda_k  
		\red{\,+\, \fnts \Gamma_{y,k}^\top \fnts \lambda_k} 
		\blue{\,+\, \xi_* \nabla_{y}\gamma_k^\top }
		+ \fntb w_{k-1}, \\
	\fnts \Lambda_{k,i} &= 
		\Delta t \, \fntb J_{k,i}^\top \sum_{j=1}^s a_{ji} \fnts \Lambda_{k,j} 
		+ \gamma_k b_i \Delta t\, \fntb J_{k,i}^\top \fnts \lambda_k  
		\red{ \,+\, \fnts \Gamma_{Y,k}^\top \fnts \lambda_k} 
		\blue{\,-\, \xi_* \nabla_{Y}\gamma_k^\top}
		+ \fntb W_{k,i}, \quad i=1,...,s,
\end{split}\right.
\]
with scalar
\[
	\xi_* \DEF \Delta t\sum_{i=1}^s \sum_{j=1}^s a_{ji} \fntb F_{K,i} \fnts \Lambda_{K,j} = \fntb F_K^\top \fntb A^\top \fnts \Lambda_K.
\]

\newpage

\bibliography{refs}

\end{document}